\documentclass[10pt,reqno]{amsart}
\usepackage[svgnames,table]{xcolor}
\usepackage[colorlinks=true,citecolor=magenta,linkcolor=cyan]{hyperref} 
\usepackage{enumerate}
\usepackage{mathabx}
\usepackage{amssymb,amsfonts, amsmath}
\usepackage{hyperref}
\usepackage{enumitem}
\usepackage[title]{appendix}
\newtheorem*{question*}{Question}
\newtheorem{theorem}{Theorem}[section]
\newtheorem{lemma}[theorem]{Lemma}
\newtheorem{proposition}[theorem]{Proposition}
\newtheorem{corollary}[theorem]{Corollary}

\newtheorem{exmp}[theorem]{\normalfont{\em{Example}}}
\newtheorem{remarks}[theorem]{\normalfont{\em{Remarks}}}
\DeclareMathDelimiter{\langle}{\mathopen}{symbols}{"68}{largesymbols}{"0A}
\DeclareMathDelimiter{\rangle}{\mathclose}{symbols}{"69}{largesymbols}{"0B}
\DeclareMathSymbol{\cap}{\mathbin}{symbols}{"5C}
\DeclareMathSymbol{\cup}{\mathbin}{symbols}{"5B}
\DeclareMathSymbol{\infty}{\mathord}{symbols}{"31}
\makeatletter
\renewcommand*\env@matrix[1][\arraystretch]{%
\edef\arraystretch{#1}%
\hskip -\arraycolsep
\let\@ifnextchar\new@ifnextchar
\array{*\c@MaxMatrixCols c}}
\makeatother
\AtEndDocument{\bigskip{\footnotesize%
\textsc{Department of Mathematics, Jagiellonian University, ul. Prof. S. Łojasiewicza 6, 30-348 Kraków, Poland} \par  
 \textit{E-mail address}: \texttt{konstantinos.tsouvalas@uj.edu.pl}}}

\title[nonArchimedean discrete groups and quasi-isometric doubles of free groups]{nonArchimedean discrete groups and quasi-isometric doubles of free groups}

\author{Konstantinos Tsouvalas}

\begin{document}

\frenchspacing

\maketitle
\thispagestyle{empty}
\begin{abstract} For any nonarchimedean local field $k$, we exhibit the first examples of Zariski-dense, discrete subgroups of special linear groups over $k$, which are not commensurable to any lattice in the points of a semisimple group over a local field, and contain an infinite index subgroup isomorphic to a lattice in the \hbox{$k$-points of a higher rank $k$-semisimple group.}
\par We prove that any Gromov-hyperbolic double of a  non-elementary free group along a cyclic subgroup embeds as a quasi-isometrically embedded subgroup of $\mathsf{SL}_5(k')\times \mathsf{SL}_5(k')$, where $k'$ is a discretely valued field. More precisely, $k'$ is the field of rational functions $k(t)$ equipped with the Gauss valuation and $k$ is any nonarchimedean local field of characteristic not equal to $2$ or $3$.\end{abstract}

\section{Introduction}

Let $k$ be a discretely valued field equipped with the absolute value induced by the discrete valuation on $k$. We equip $\mathsf{SL}_n(k)$ (and any of its $k$-subgroups) with the $\ell_{\infty}$-norm induced by the absolute value on $k$. Recall that a subgroup $\Gamma$ of $\mathsf{SL}_n(k), n\geq 2$, is called {\em discrete} if it is discrete with respect to the subspace topology induced by the matrix norm. The subgroup $\Gamma$ of $\mathsf{SL}_n(k)$ is {\em metrically proper} if for any positive integer $L$ there are finitely many elements of $\Gamma$ of norm at most equal to $L$. In the latter case $\Gamma$ acts isometrically and properly on the Bruhat--Tits building of $\mathsf{SL}_n(k)$, see \cite{BT72}. 

In this paper, we provide new constructions of metrically proper subgroups of $k$-matrix groups when $k$ is a discretely valued field. Among the finitely generated groups that are not abstractly commensurable to $S$-arithmetic groups, those that were known prior to this work to admit discrete and faithful matrix representations over nonarchimedean local fields were, up to finite-index, free products of free abelian groups and were exhibited by Douba, Kubrak and the author of the present paper in \cite{DKT}. These groups are constructed as thin subgroups of any lattice in $\mathsf{SL}_3(k)$. When $k$ is non-locally compact, the property of a subgroup of a $k$-semisimple group being discrete is weaker\footnote{For example, for any field $k_0$, $\mathsf{SL}_n(k_0)$ embeds as a discrete and bounded subgroup of $\mathsf{SL}_n(k_0(t))$, where $k_0(t)$ is equipped with the discrete valuation sending $f_1(t)/f_2(t)$ to $\textup{deg}f_2-\textup{deg}f_1$.} than being metrically proper. In the non-locally compact setting, metrically proper examples include any discrete subgroup over a nonarchimedean local field (seen as a subgroup of a matrix group over a non-local extension) and closed surface groups \cite[Cor. 1.4]{BIPP-PLMS}. 

 We prove that the classes of discrete and metrically proper subgroups over any infinite discretely valued field is closed under taking free products.

\begin{theorem}\label{mainthm-1} Let $k$ be a discretely valued field and $\Gamma_1,\Gamma_2$ be discrete \textup{(}resp. metrically proper\textup{)} subgroups of $\mathsf{SL}_n(k)$. The free product $\Gamma_1\ast \Gamma_2$ embeds as a discrete \textup{(}resp. metrically proper\textup{)} subgroup of $\mathsf{SL}_{d+1}(k)$, $d=\binom{4n^2+2n}{2n}$.
\end{theorem}

Theorem \ref{mainthm-1} provides new examples of discrete and metrically proper groups over discretely valued fields. When $k$ is a local field the theorem above can be viewed as the nonarchimedean analogue of the combination theorem of Danciger--Gu\'eritaud--Kassel in \cite[Thm. 1.6]{DGKcomb} for discrete subgroups of real matrix Lie groups. 

\enlargethispage{1\baselineskip}

As a consequence of Theorem \ref{mainthm-1} and its proof, we also establish the following theorem.

\begin{theorem}\label{thm-embd} Let $k$ be a nonarchimedean local field. For any discrete subgroup $\Gamma$ of $\mathsf{SL}_n(k)$ there exist:
\begin{enumerate}\setlength{\itemsep}{1.2pt}
\item an embedding $\tau:\mathsf{SL}_n(k)\rightarrow \mathsf{SL}_{d+1}(k)$, where $d=\binom{n^2+n}{n}$, and
\item  a Zariski-dense infinite covolume subgroup $\Gamma'$ of $\mathsf{SL}_{d+1}(k)$
\end{enumerate} with the property that $\tau(\Gamma)=\Gamma' \cap \tau(\mathsf{SL}_n(k))$ is a subgroup of  infinite index in $\Gamma'$. Moreover, $\Gamma'$ is isomorphic to the free product of $\Gamma$ with a finite-rank free group. \end{theorem}

For any nonarchimedean local field $k$, Theorem \ref{thm-embd} applies to any lattice in the $k$-points of a semisimple $k$-group. This provides the first-known examples of Zariski-dense discrete subgroups of $\mathsf{SL}_{d+1}(k)$, which are not commensurable to any lattice in a semisimple group over a local field, containing an infinite index subgroup isomorphic to a lattice in the $k$-points of a higher rank $k$-semisimple group. These are new examples of finitely generated groups acting on locally compact Euclidean buildings, namely the Bruhat--Tits building of $\mathsf{SL}_{d+1}(k)$. As an explicit example, for any prime $p$, $\mathsf{SO}_n(\mathbb{Z}[\frac{1}{p}])\ast \mathbb{Z}$, $n\geq 3$, embeds as a subgroup of $\mathsf{SL}_m(\mathbb{Z}[\frac{1}{p}])$, for some $m$ depending only on $n$ (see Corollary \ref{SO}).

Motivated by Nori's problem (see \cite[Q. 8.3]{Fisher}), Chatterji--Venkataramana in \cite[Q.2]{CV-lattice} asked whether a Zariski-dense discrete subgroup  of a simple real algebraic group $G$, intersecting a higher-rank simple Lie subgroup of $G$ in a lattice, must be a lattice in $G$. While there is a positive answer\footnote{If the smaller Lie subgroup is not semisimple and in particular a horospherical subgroup of $G$, a positive answer to Nori's problem was proved in \cite{Benoist-Miquel, Oh}.} in many cases (see \cite[Thm. 1 \& 3]{CV-lattice}), a negative answer was subsequently obtained by Danciger--Gu\'eritaud--Kassel in \cite[Thm. 1.6 \& Cor. 1.9]{DGKcomb}. Theorem \ref{thm-embd} establishes counterxamples to the nonarchimedean analogue of \cite[Q.2]{CV-lattice}.

As an application of Theorem \ref{mainthm-1} we prove uniform linearity of free products of linear groups over the polynomial ring of any integral domain.

 \begin{corollary}\label{R[t]} For any integral domain $\mathsf{R}$, the free product $\mathsf{SL}_n(\mathsf{R}[t]) \ast \mathsf{SL}_n(\mathsf{R}[t])$ embeds as a subgroup of $\mathsf{SL}_{d+1}(\mathsf{R}[t])$, $d=\binom{4n^2+2n}{2n}$. \end{corollary}

The above linearity result can be viewed as a nonarchimedean analogue of the $\mathbb{Z}$-linearity result \cite[Cor. 1.10]{DGKcomb}.
\subsection*{Nonarchimedean quasi-isometric embeddings of doubles of free groups} 
For a group $\Gamma$ and a subgroup $\Lambda$ of $\Gamma$, the {\em double of $\Gamma$ along $\Lambda$} is the amalgam $\Gamma\ast_{\Lambda}\Gamma$ of two copies of $\Gamma$ along $\Lambda$. For a local field $k$, the Gauss valuation on the field of rational functions $k(t)$ extends the discrete valuation on $k$ (see \S \ref{Gauss}). Equip $\mathsf{SL}_d(k(t))$ with the $\ell_{\infty}$-norm induced by the Gauss valuation on $k(t)$.  A finitely generated subgroup $\Gamma$ of $\mathsf{SL}_n(k(t))\times \mathsf{SL}_n(k(t))$ is {\em quasi-isometrically embedded} \hbox{if there exist $c,a>0$} such that for every $\gamma\in \Gamma$, $$||\gamma ||\cdot ||\gamma^{-1}||\geq e^{a|\gamma|_{\Gamma}-c}$$ where $||\gamma^{\pm 1}||$ is the norm of $\gamma^{\pm 1}$ (seen as a matrix in $\mathsf{SL}_{2n}(k)$) and $|\cdot|_{\Gamma}$ is a word length function induced by a choice of word metric on $\Gamma$. Having the previous conventions we obtain the following.

\begin{theorem}\label{qie-thm} Let $k$ be a nonarchimedean local field, of characteristic not equal to $2$ or $3$, and $\Gamma$ be the double of a finite-rank free group along a maximal cyclic subgroup. Then $\Gamma$ embeds as a quasi-isometrically embedded subgroup of $\mathsf{SL}_{5}(k(t))\times \mathsf{SL}_5(k(t))$.\end{theorem}

Theorem \ref{qie-thm} provides examples of one-ended Gromov-hyperbolic groups (e.g. closed surface groups) acting on the Bruhat--Tits building of $\mathsf{SL}_{5}(k(t))\times \mathsf{SL}_5(k(t))$ with orbit maps that are quasi-isometric embeddings. Moreover, we prove (see Theorem \ref{prod-qie}) that any double of a free group along a maximal cyclic subgroup embeds as a metrically proper subgroup of $\mathsf{SL}_2(k(t))\times \mathsf{SL}_2(k(t))$. All these constructions of metrically proper groups arise by amalgamating two conjugates of a quasi-isometric free subgroup of $\mathsf{SL}_2(k)$ in the ambient product. To the author's knowledge, Theorem \ref{qie-thm} provides the first-known examples of non-virtually Fuchsian\footnote{i.e hyperbolic groups that are not up to finite-index a free or a surface group}, metrically proper hyperbolic subgroups of a matrix group over a discretely valued field. In contrast, the existence of discrete, non-virtually free, hyperbolic subgroups into some $k$-semisimple group over a nonarchimedean local field $k$ remains an open problem (see also \cite[Q.1]{FLSS}). Prior to the aforementioned results, the only known examples of metrically proper, non-virtually free hyperbolic groups over discretely valued fields were the quasi-isometrically embedded surface subgroups of $\mathsf{SL}_n(\mathbb{R}(X))$\footnote{$\mathbb{R}(X)$ is equipped with the discrete valuation sending $f_1(X)/f_2(X)$ to $\textup{deg}f_2-\textup{deg}f_1$}, $n\geq 3$, from \cite[Cor. 1.14]{BIPP-PLMS} and \cite[Ex. 1.15 \&  1.16]{BIPP-PLMS}, obtained \hbox{as degenerations of Hitchin representations in $\mathsf{SL}_n(\mathbb{R})$.}

\section{Preliminary lemmas}

Throughout this article, $k$ denotes a valued field equipped with a nontrivial, nonarchimedean discrete valuation $\textup{val}:k\rightarrow \mathbb{Z}\cup \{\infty\}$. We equip $k$ with the absolute value  $$|x|=e^{-\textup{val}(x)}, \ x\in k.$$ The absolute value on $k$ satisfies the ultrametric inequality \begin{align}\label{ultrametric} \big|a_1+\cdots+a_r\big|\leq \max \big\{|a_1|,\ldots, |a_r|\big\} \ \  \forall \  a_1,\ldots,a_r \in k.\end{align} Recall that if $k$ is local, then it is isomorphic to a finite extension of either the $p$-adic numbers $\mathbb{Q}_p$ or the field of formal Laurent series $\mathbb{F}_q((\frac{1}{t}))$ over the finite field $\mathbb{F}_q$.

The vector space $k^n$ is equipped with the canonical basis $(e_1,\ldots,e_n)$ and the norm $||\cdot||:k^n\rightarrow \mathbb{R}_{+}$ defined as follows $$||a_1e_1+\cdots+a_ne_n||=\max_{1\leq j\leq n}|a_j|.$$ For any matrix $g=(g_{ij})_{i,j=1}^{n}$ with entries in $k$ its $\ell_{\infty}$-norm is $$||g||=\max_{||v||=1}||gv||=\max_{1\leq i,j\leq n}|g_{ij}|.$$

Denote by $\mathbb{P}(k^n)$ the projective space of $k^n$. For a unit vector $u\in k^n$ let $$B_{\epsilon}([u]):=\big\{[\omega]\in \mathbb{P}(k^n): ||\omega||=1, ||\omega-u||\leq \epsilon\big\}$$ be the ball of radius $\epsilon>0$ with center $[u]\in \mathbb{P}(k^n)$.  For two vectors $u=(u_1,\ldots,u_n)^\textup{t}\in k^n$,  $v=(v_1,\ldots,v_n)^\textup{t}\in k^n$ define their dot product as $$u\cdot v:=u_1v_1+\cdots+u_nv_n.$$ By the ultrametric property, the dot product  satisfies the Cauchy--Schwarz inequality $$|u\cdot v|\leq ||u||\cdot ||v||.$$ 

For a unit vector $v\in k^{n}$ consider the $(n-1)$-hyperplane $$v^{\perp}:=\big\{\omega \in k^n: \omega \cdot v=0\big\}$$ and define $\mathcal{N}_{\epsilon}(v^{\perp}):=\bigcup_{x \in \mathbb{P}(v^{\perp})}B_{\epsilon}(x)$ for any $\epsilon>0$.

 For $m \in \mathbb{N}$, the $m$-symmetric power of $k^n$ is the $k$-vector space,  $$\textup{Sym}^m k^n:=\bigoplus_{s_1+\cdots+s_n=m} k e_{1}^{s_1}\cdots e_{n}^{s_n}$$ and the representation $\textup{sym}^{m}:\mathsf{GL}_n(k) \rightarrow \mathsf{GL}(\textup{Sym}^{m}k^n\big)$ is defined as follows for $g \in \mathsf{GL}_n(k)$, $$\textup{sym}^m (g)\big(e_1^{s_1}\cdots e_{n}^{s_n}\big):=(ge_1)^{s_1}\cdots (ge_n)^{s_n}$$ for any basis vector $e_1^{s_1}\cdots e_n^{s_n}$ of $\textup{Sym}^m k^n$. Note that if $g\in \mathsf{SL}_n(k)$ then $\textup{det}(\textup{sym}^m(g))=1$ and $\textup{sym}^m(\mathsf{SL}_n(k))$ is a subgroup of $\mathsf{SL}(\textup{Sym}^mk^n)$.

For an unbounded subring $A$ of $k$ define $$v(A):=\min\{ |\textup{val}(a)|:a\in A, |a|>1\}.$$

\begin{lemma}\label{lemma} Let $A$ be an unbounded subring of $k$ and fix  integers $n\geq 2$, $m\geq n^2 v(A)$. There exists a faithful representation $\rho_{nm}:\mathsf{SL}_n(A)\rightarrow \mathsf{SL}_d(A)$, $d=\binom{n+m}{n}$, and vectors $v_1,v_2\in A^d$ such that for every $g\in \mathsf{SL}_d(A)$ with $||g||\geq 2|v_1\cdot v_2|$ the  bound holds $$| (\rho_{nm}(g)v_1-v_1)\cdot v_2|\geq  ||\rho_{nm}(g)||.$$

\end{lemma}

\begin{proof}  Let $j_n:\mathsf{SL}_n(k)\rightarrow \mathsf{SL}_{n+1}(k)$ be the embedding $$j_n(g)=\begin{pmatrix} 1 & \\ & g\end{pmatrix}, \  \ g\in \mathsf{SL}_n(k)$$ fixing the basis vector $e_{n+1}$ in the decomposition $k^{n+1}=ke_{n+1}\oplus k^n$. 

Consider the standard ordered basis of $\textup{Sym}^{m}k^{n+1}$ whose first $n$-tuple of vectors is $\big(e_1^m,\ldots,e_n^m\big)$ while the rest are of the form $\{e_{1}^{s_1}\cdots e_n^{s_{n+1}}:s_1+\cdots+s_{n+1}=m\}$. Thus, we obtain an identification $\textup{Sym}^{m}k^{n+1}\simeq k^{d}$, $d=\binom{n+m}{n}$, and a splitting $k^{d}=k^{n}\oplus V_{d}$, where $k^n=ke_1^m\oplus\cdots \oplus ke_n^m$ and $V_{d}\simeq k^{d-n}$. With respect to this decomposition of $k^d$, for the composition \hbox{$\rho_{nm}:\mathsf{SL}_n(k)\rightarrow \mathsf{SL}_d(k)$}, $\rho_{nm}:=\textup{sym}^m\circ j_{n}$, and any $g=(g_{ij})_{i,j=1}^{n}$ in  $\mathsf{SL}_{n}(k)$, $$\rho_{nm}(g)=\begin{pmatrix}[1.1] \mathcal{X}_{nm} (g) & \mathcal{Y}_{nm}(g) \\ \mathcal{Z}_{nm}(g)  & \mathcal{T}_{nm}(g)  \end{pmatrix}$$ where $\mathcal{X}_{nm}(g)$ (resp. $\mathcal{T}_{nm}(g)$) are $n\times n$ (resp. $(d-n)\times (d-n)$) matrices with \hbox{entries in $k$ and} $$\mathcal{X}_{nm}(g) :=\begin{pmatrix}[1.2] g_{11}^m & \cdots & g_{1n}^m \\ \vdots & \ddots & \vdots \\ g_{n1}^m & \cdots &g_{nn}^m \end{pmatrix}.$$

Since the kernel of the representation $\textup{sym}^m:\mathsf{SL}_n(k)\rightarrow \mathsf{SL}(\textup{Sym}^mk^n)$ contains only scalar multiples of the identity, $\rho_{nm}$ is a faithful representation and by the definition of $\rho_{nm}$, $\rho_{nm}(\mathsf{SL}_n(A))$ is a subgroup of $\mathsf{SL}_d(A)$. Fix $\pi\in A$ with $v(A)=\textup{val}(\pi)$ and define \begin{align}\label{vect1}v_1&:=\pi e_1+\pi^2e_2+\cdots +\pi^{n}e_n \in A^d\\ \label{vect2} v_2&:= \pi^n e_1+\pi^{2n}e_2+\cdots +\pi^{n^2}e_n\in A^d.\end{align} For every $g\in \mathsf{SL}_n(k)$,  $$ \rho_{nm}(g)v_1=\mathcal{X}_{nm}(g) v_1+\mathcal{Z}_{nm}(g) v_1=\sum_{i=1}^{n}\Bigg( \sum_{j=1}^{n}g_{ij}^m \pi^{j}\Bigg)e_i+\underbrace{\mathcal{Z}_{nm}(g) v_1}_{\in V_d}.
$$ By using the previous equality, for every $g\in \mathsf{SL}_n(k)$, \begin{align}\label{equation2} |\rho_{nm}(g)v_1\cdot v_2 |=\Bigg|\sum_{i,j=1}^{n}g_{ij}^m\pi^{ni+j} \Bigg| =\Bigg|\sum_{(i,j):g_{ij}\neq 0}g_{ij}^m\pi^{ni+j} \Bigg|.\end{align} Observe that when $1\leq i,j,s,\ell \leq n$, $(i,j)\neq (s,\ell)$ and $g_{ij}g_{s\ell}\neq 0$, then \begin{align*}\frac{|\pi|^{ni+j}}{{|\pi|^{ns+\ell}}}&=e^{(n(i-s)+j-\ell)v(A)}\neq 1.\end{align*} Since $m\geq n^2 v(A)$ cannot divide the non-zero integer $(n(i-s)+j-\ell)v(A)$ (the latter is in $\{(1-n^2)v(A),\ldots, (n^2-1)v(A)\}$), $\pm \textup{val}(\frac{g_{s\ell}}{g_{ij}})m \neq (n(i-s)+j-\ell)v(A)$ and hence $$|\pi^{ni+j}g_{ij}^m|\neq |\pi^{ns+\ell}g_{s\ell}^m|,$$ \hbox{when $(i,j)\neq (s,\ell)$ and $g_{ij}g_{s\ell}\neq 0$. Since the absolute value $|\cdot|$ on $k$ is an ultrametric,  (\ref{equation2}) shows} \begin{align*}|(\rho_{nm}(g)v_1)\cdot v_2|=\max_{1\leq i,j\leq n} |\pi^{ni+j}g_{ij}^m|&\geq |\pi|^{n+1} \max_{1\leq i,j\leq n}|g_{ij}|^m=|\pi|^{n+1} ||g||^m.\end{align*}  By the definition of the $m$-symmetric power $\textup{sym}^m$ and the choice of the base of $\textup{Sym}^mk^{n+1}$,  $$| |\textup{sym}^m(j_{n}(g))| |=||\rho_{nm}(g)||=||g||^m, \ g\in \mathsf{SL}_n(k).$$ Thus, for any $g\in \mathsf{SL}_n(k)$ the following bound holds \begin{align}\label{lem1-mainbd} | (\rho_{nm}(g)v_1)\cdot v_2|\geq |\pi|^{n+1} ||\rho_{nm}(g)||.\end{align} 
 For every $g\in \mathsf{SL}_n(k)$ with $||g||\geq 2|v_1\cdot v_2|$, (\ref{lem1-mainbd}) implies $$|(\rho_{nm}(g)v_1-v_1)\cdot v_2|\geq |\pi|^{n+1} ||g||^m-|v_1\cdot v_2|\geq  ||g||^m= | |\rho_{nm}(g)| |.$$ Finally, the restriction of $\rho_{nm}$ on $\mathsf{SL}_n(A)$ satisfies the properties of the lemma. \end{proof}

For $e_i\in k^n$ recall that $e_i^{\perp}=\textup{span}\{e_j:j\neq i\}$. We will need the following folklore lemma.

\begin{lemma}\label{order2} Let $A$ be an unbounded subring of $k$. For any $0<\epsilon<1$ and $T>1$, there is $w\in \mathsf{SL}_r(A)$, $r\geq 3$, with the property that $$w^{\pm 1} \big(\mathbb{P}(k^r)\smallsetminus \mathcal{N}_{\epsilon}(e_{r}^{\perp})\big)\subset B_{\epsilon}([e_1])$$ and also $||w^{\pm 1} \omega ||\geq T||\omega||$ for every $[\omega]\in \mathbb{P}(k^r)\smallsetminus \mathcal{N}_{\epsilon}(e_r^{\perp})$.\end{lemma}

\begin{proof} For any $\omega=\omega_1e_1+\cdots +\omega_r e_r\in k^r$ with $||\omega||=1$ and $[\omega]\in \mathbb{P}(k^r)\smallsetminus \mathcal{N}_{\epsilon}(e_{r}^{\perp})$ we have  $|\omega_r|\geq \epsilon$. Fix $a\in A$ with $|a|\geq {\epsilon^{-6}}T$ and consider the order $2$ matrix $w_a\in \mathsf{SL}_r^{\pm}(A)$, $$w_a:=\begin{pmatrix} a & & & a^2-1 \\  & -1 & & \\ & &  \textup{I}_{r-3}  &  \\ -1  & & & -a \end{pmatrix}.$$  Then observe that \begin{align*}\frac{w_a \omega}{(a^2-1)\omega_r+a\omega_1}&= e_1-\frac{\omega_2}{(a^2-1)\omega_r+a\omega_1}e_2+\cdots-\frac{\omega_1+a\omega_r}{(a^2-1)\omega_r+a\omega_1}e_r,\\ | |w_{a}\omega||&=\big|(a^2-1)\omega_r+a\omega_1\big|=|a^2\omega_r|\geq T, \ \ \textup{hence}\end{align*}$$\Bigg|\Bigg| \frac{w_a\omega}{(a^2-1)\omega_r+a\omega_1}-e_1\Bigg|\Bigg|\leq \frac{|a|}{|a^2\omega_r|}\leq \epsilon$$ and $[w_a\omega]\in B_{\epsilon}([e_1])$. The lemma follows.\end{proof}

 For a finitely generated group $\Delta$ a representation $\rho:\Delta\rightarrow \mathsf{SL}_n(k)$ is a {\em quasi-isometric embedding} if there exist $c,a>0$ such that for every $g\in \Delta$, $$||\rho(g)||\cdot ||\rho(g^{-1})||\geq  e^{a|g|_{\Delta}-c}$$ where $|g|_{\Delta}$ denotes the word length of $g$ with respect to a fixed word metric on $\Delta$. In particular, if $\mathcal{BT}_n$ denotes the the Bruhat--Tits building of $\mathsf{SL}_n(k)$, equipped with the apartment metric $\mathsf{d}_{\mathcal{A}}$, the orbit map $(\Delta, |\cdot|_{\Delta})\rightarrow (\mathcal{BT}_n,\mathsf{d}_{\mathcal{A}})$, $\gamma \mapsto \rho(\gamma)x_0$ (where $x_0\in \mathcal{BT}_n$ is a fixed basepoint\footnote{For $g\in \mathsf{SL}_n(k)$, $|d_{\mathcal{A}}(gx_0,x_0)-\log(||g||\cdot ||g^{-1}||)|$ is bounded above  by a constant independent of $g$.}) is a quasi-isometric embedding of metric spaces. A subgroup $\Delta$ of $\mathsf{SL}_n(k)$ is quasi-isometrically embedded if it is finitely generated and its inclusion in $\mathsf{SL}_n(k)$ is a quasi-isometric embedding.

We also need the following ping-pong combination of discrete subgroups of $\mathsf{SL}_d(k)$.

\begin{proposition}\label{theorem-1} Let $A$ be an unbounded subring of $k$ and $v_1,v_2\in A^d$ defined in (\ref{vect1}) and (\ref{vect2}) respectively. Let $\Gamma_1,\Gamma_2$ be subgroups of $\mathsf{SL}_d(A)$ such that for every $g\in \Gamma_i\smallsetminus \{\textup{I}_d\}$, \begin{align}\label{ineq1-thm1}|(\gamma v_1-v_1)\cdot v_2|\geq ||\gamma||.\end{align} There exists $w\in \mathsf{SL}_{d+1}(A)$ such that the subgroup of $\mathsf{SL}_{d+1}(A)$, \begin{align}\label{gengp} \Bigg\langle \begin{pmatrix} 1& \\ & \Gamma_1  \end{pmatrix}, w\begin{pmatrix} 1& \\ & \Gamma_2  \end{pmatrix}w^{-1} \Bigg\rangle\end{align} is a discrete subgroup of $\mathsf{SL}_{d+1}(k)$ isomorphic to $\Gamma_1\ast \Gamma_2$. Moreover, if $\Gamma_1$ and $\Gamma_2$ are metrically proper \textup{(}resp. quasi-isometrically embedded\textup{)} in $\mathsf{SL}_d(k)$, then the group defined in (\ref{gengp}) is also metrically proper \textup{(}resp. quasi-isometrically embedded\textup{)}. \end{proposition}

\begin{proof} Recall that there is $\pi\in A$ such that $v_1=\pi e_1+\cdots+\pi^ne_n, v_2=\pi^{n}e_1+\cdots+\pi^{n^2}e_n$ are defined in (\ref{vect1}) and (\ref{vect2}) respectively. 
We first claim there is $h\in \mathsf{SL}_{d+1}(A)$ such that \begin{align*}[(1,v_1)^\textup{t}]&=[he_1],\\ [(-v_1\cdot v_2,v_2)^\textup{t}]&=[(h^{-1})^\textup{t} e_{d+1}].\end{align*}

To see this, consider $\sigma_1,\sigma_2\in \mathsf{SL}(A\oplus A^d)$ defined as follows $$ \sigma_1:=\begin{pmatrix}
1 &  &  &  &  & \\
-\pi & 1  &  &  &  & \\
-\pi^2 & 0 & 1 &  &  \\
\vdots & \vdots & \vdots  &\ddots   & &  \\
 -\pi^n& 0 & 0 & \cdots  & 1 \\
 &  &  &  &  &  \textup{I}_{d-n} \\
\end{pmatrix},\  \sigma_2:=\begin{pmatrix}
1 &  0 & 0& \cdots & 0 & \\
 & 1 & \pi^{n} & \cdots & \pi^{n(n-1)} & \\
 &  & 1 & \cdots &0 &  \\
 &   &  & \ddots  & \vdots &  \\
 &  &  &  & 1 & \\
 &  &  &  &  &  \textup{I}_{d-n} \\
\end{pmatrix} $$ with respect to the ordered basis $(e_{d+1},e_1,\ldots, e_n,\ldots,e_d)$ of $A\oplus A^d$.  By the definition of $v_1,v_2\in A^d$ we have  $v_1\cdot v_2=\pi^{n+1}+\pi^{2(n+1)}+\cdots +\pi^{n(n+1)}$ and also \begin{align*}[\sigma_2\sigma_1 (1,v_1)^\textup{t} ] &=[e_{d+1}]\\ 
[(\sigma_2^{-1})^{\textup{t}}(\sigma_1^{-1})^\textup{t} (-v_1\cdot & v_2,v_2)^\textup{t} ] =[e_1].\end{align*} Thus, the claim follows by taking the matrix $h \in \mathsf{SL}(A\oplus A^d)\simeq \mathsf{SL}_{d+1}(A)$, $$h:=\sigma_1^{-1}\sigma_2^{-1}\begin{pmatrix} & -1 & \\ 1 & &\\ & & \textup{I}_{d-1}\end{pmatrix}.$$

Now consider the embedding $j_d:\mathsf{SL}_d(k)\rightarrow \mathsf{SL}_{d+1}(k)$, $$j_d(g):=\begin{pmatrix}1 & \\ & g\end{pmatrix},  \ g\in \mathsf{SL}_d(k).$$ 
Given $h\in \mathsf{SL}_{d+1}(A)$, for every $g\in \Gamma_i\smallsetminus \{\textup{I}_d\}$ \hbox{we have that} \begin{align}\nonumber \frac{|(h^{-1}j_d(g)he_1)\cdot e_{d+1}|}{||j_d(g)||}&=\frac{|(j_d(g)he_1)\cdot ((h^{-1})^te_{d+1})|}{||j_d(g)||}=\frac{|(j_d(g) (1,v_1)^{\textup{t}})\cdot (-v_1\cdot v_2,v_2)^{\textup{t}}|}{||j_d(g)||}\\
\label{inclusion-a} &= \frac{|(g v_1-v_1)\cdot v_2|}{||g||}\geq 1.\end{align}  Fix any  $0<\theta< (10||h||\cdot ||h^{-1}||)^{-2}$. By (\ref{inclusion-a}), for any $g\in \Gamma_i\smallsetminus \{\textup{I}_d\}$, \begin{align}\label{inclusion1}\big(h^{-1}j_d(g)h\big)B_{\theta}([e_1])\subset \mathbb{P}(k^{d+1})\smallsetminus \mathcal{N}_{\theta}\big(e_{d+1}^{\perp}\big)\\ 
\label{inclusion-b} | |(h^{-1}j_d(g)h)\omega| |\geq \theta | |j_d(g)| |\cdot ||\omega||, \ \forall \ [\omega]\in B_{\theta}([e_1]).\end{align}

By Lemma \ref{order2}, given $\theta>0$, there is $w\in \mathsf{SL}_{d+1}(A)$ such that \begin{align}\label{inclusion3''} w^{\pm 1}\Big(\mathbb{P}(k^{d+1})\smallsetminus \mathcal{N}_{\theta}(e_{d+1}^{\perp})\Big)&\subset  B_{\theta}([e_1])\\ \label{inclusion3'''} ||w^{\pm 1} \omega||\geq \theta^{-7} ||\omega||\end{align} for every $[\omega]\in \mathbb{P}(k^{d+1})\smallsetminus \mathcal{N}_{\theta}(e_{d+1}^{\perp})$. By combining (\ref{inclusion1}), (\ref{inclusion-b}), (\ref{inclusion3''}) and (\ref{inclusion3'''}), we have \begin{align}
\label{inclusion3} (wh) j_d(\gamma_2)(wh)^{-1} \big(\mathbb{P}(k^{d+1})\smallsetminus \mathcal{N}_{\theta}(e_{d+1}^{\perp})\big)&\subset B_{\theta}([e_1]), \ \forall \  \gamma_2\in \Gamma_2\smallsetminus \{\textup{I}_d\}\\ \nonumber | |(wh) j_d(\gamma_2)(wh)^{-1} \omega | |&\geq \theta^{-7}| |(hj_d(\gamma_2)h^{-1})w^{-1}\omega| |\\ & \nonumber \geq \theta^{-6}| |j_d(\gamma_2)| |\cdot | |w^{-1}\omega| |\\ \label{inclusion3'} &\geq \theta^{-13} | |j_d(\gamma_2)| |\cdot ||\omega||\end{align} for any $[\omega]\in \mathbb{P}(k^{d+1})\smallsetminus \mathcal{N}_{\theta}(e_{d+1}^{\perp})$ and $\gamma_2\in \Gamma_2\smallsetminus \{\textup{I}_d\}$.

Now let $\psi_{d}:\Gamma_1\ast \Gamma_2\rightarrow \mathsf{SL}_d(k)$ be the unique representation satisfying $$\psi_d(\gamma):= \left\{\begin{matrix}
hj_d(\gamma)h^{-1} & \gamma\in \Gamma_1  \\
(wh)j_d(\gamma)(wh)^{-1} & \gamma\in \Gamma_2  \\
\end{matrix}\right..$$ Clearly, as $wh\in \mathsf{SL}_{d+1}(A)$, $\psi_d(\Gamma_1\ast \Gamma_2)$ is a subgroup of $\mathsf{SL}_{d+1}(A)$. By using the ping-pong inclusions (\ref{inclusion1}), (\ref{inclusion3}) and the bounds  (\ref{inclusion-a}), (\ref{inclusion-b}) and (\ref{inclusion3'}), for any reduced word $g_1\cdots g_{\ell}\in \Gamma_1\ast \Gamma_2$, where $g_i \neq \textup{I}_n$, we obtain the bound \begin{align}\label{final-bd} || \psi_d(\gamma_1\cdots \gamma_{\ell}) | |\geq \theta^{\frac{13(1-\ell)}{2}} \theta^{\frac{\ell+1}{2}} ||j_d(\gamma_1)| |\cdots | |j_d(\gamma_{\ell}) | |=  \theta^{7-6\ell} | |\gamma_1| |\cdots | |\gamma_{\ell}| |.\end{align} 

It remains to verify that $\psi_d$ has discrete image. First, note by (\ref{ineq1-thm1}) that \begin{align}\label{G12-discrete}\frac{||\gamma-\textup{I}_d||}{||\gamma||}\geq \frac{1}{||v_1||\cdot ||v_2||} \ \ \forall \ \gamma \in \Gamma_i \smallsetminus \{\textup{I}_d\}.\end{align}  Let us set \begin{align*}\eta:&= \big((1+||v_1||\cdot ||v_2||)\cdot ||h||\cdot ||h^{-1}||\cdot ||w||\cdot ||w^{-1}||\big)^{-1},\\ \mathcal{S}&:=\big\{g\in \Gamma_1\ast \Gamma_2: ||\psi_d(g)-\textup{I}_{d+1}||\leq \eta\big\}.\end{align*} By (\ref{final-bd}),  $\mathcal{S}$ is a subset of $\Gamma_1\cup \Gamma_2$ since any element of $(\Gamma_1\ast \Gamma_2)\smallsetminus \Gamma_1\cup \Gamma_2$ has norm at least  $\frac{1}{\theta^5}$, while any matrix in $\mathcal{S}$ has norm equal to $1$. If $||\psi_d(\gamma)-\textup{I}_{d+1}||\leq \eta$ for some $\gamma \in \Gamma_1\cup \Gamma_2$, by the sub-multiplicativity of the matrix norm it follows that $$| |\gamma-\textup{I}_{d}| |\leq \big(||w||\cdot ||w^{-1}||\cdot ||h||\cdot ||h^{-1}||\big) | |\psi_d(\gamma)-\textup{I}_{d+1} | |\leq \frac{1}{1+||v_1||\cdot ||v_2||}.$$ Thus  (\ref{G12-discrete}) implies $\mathcal{S}=\{\textup{I}_{d+1}\}$ and $\psi_d$ is a discrete and faithful representation.

Now we assume that $\Gamma_1$ and $\Gamma_2$ are metrically proper subgroups of $\mathsf{SL}_d(k)$. Fix $L>1$. By (\ref{final-bd}), we can check that $\{\gamma\in \Gamma_1\ast \Gamma_2: ||\psi_d(\gamma)||\leq L\}$ is contained in $$\Bigg\{\gamma_{1}\cdots \gamma_\ell: \gamma_i\in \Gamma_1\cup \Gamma_2, ||\gamma_1||\leq \frac{L}{\theta},\ldots, ||\gamma_\ell ||\leq \frac{L}{\theta}, 1\leq \ell \leq \frac{7}{6}-\frac{\log L}{6\log\theta} \Bigg\}.$$ The latter set has cardinality at most equal to $$\Big|\Big\{\gamma \in \Gamma_1\cup \Gamma_2: ||\gamma||\leq \frac{L}{\theta}\Big\}\Big|^{\frac{7}{6}-\frac{1}{6}\frac{\log L}{\log\theta}}$$ and $\psi_d(\Gamma_1\ast\Gamma_2)$ is metrically proper.

Now assume that $\Gamma_1,\Gamma_2$ are quasi-isometrically embedded in $\mathsf{SL}_d(k)$ and fix a word length function $|\cdot|:\Gamma_1\ast \Gamma_2\rightarrow \mathbb{N}$ induced by a choice of word metrics on $\Gamma_i$. There are $c,a>0$ with $$||\gamma||\cdot ||\gamma^{-1}||\geq e^{a|\gamma|-c} \ \ \forall \ \gamma \in \Gamma_i$$ and choose $0<\theta<(10e^c||h||\cdot||h^{-1}||)^{-2}$ as above before (\ref{inclusion1}). \hbox{By (\ref{final-bd}) and the choice of $\theta>0$,} $$|| \psi_d(\gamma_1\cdots \gamma_{\ell}) | |\geq  \theta^7 e^{a|\gamma_1|+\cdots +a|\gamma_\ell|}\geq \theta^7 e^{a|\gamma_1\cdots \gamma_\ell|}$$ for any reduced product $\gamma_1\cdots \gamma_\ell\in \Gamma_1\ast \Gamma_2$. Thus, in this case $\psi_d$ is a quasi-isometric embedding and the lemma follows.\end{proof}

\section{Proof of Theorem \ref{mainthm-1}}\label{proofs}

Recall that $k$ is a discretely valued field and for an unbounded subring $A$ of $k$, $v(A)=\textup{min}\{|\textup{val}(a)|:a \in A, |a|>1\}$. For $\nu \in k$ denote by $\iota_{\nu}:\mathsf{SL}_\ell (k)\rightarrow \mathsf{SL}_{2\ell}(k)$ the faithful representation defined as follows, \begin{align}\label{def-iota}\iota_{\nu}(g): =\begin{pmatrix} \textup{I}_\ell & \nu (g- \textup{I}_\ell)\\ &  g  \end{pmatrix}, \ g\in \mathsf{SL}_{\ell}(k).\end{align} We now prove the following theorem from which Theorem \ref{mainthm-1} follows readily.

\begin{theorem}\label{mainthm-1'} Let $A$ be an unbounded subring of $k$ and $\Gamma_1,\Gamma_2$ be discrete \textup{(}resp. metrically proper or quasi-isometrically embedded\textup{)} subgroups of $\mathsf{SL}_\ell(k)$, $\ell \geq 2$, contained in $\mathsf{SL}_\ell(A)$. The free product $\Gamma_1\ast \Gamma_2$ embeds in $\mathsf{SL}_{d+1}(A)$ as a discrete \textup{(}resp. metrically proper or quasi-isometrically embedded\textup{)} subgroup of $\mathsf{SL}_{d+1}(k)$, $d=\binom{4\ell^2v(A)+2\ell}{2\ell}$.\end{theorem}

\begin{proof} Let $m:=n^2v(A)$ and $n:=2\ell$. By  Lemma \ref{lemma} there exist $v_1,v_2\in A^{d}$, $d=\binom{n+m}{n}$ (defined in (\ref{vect1}) and (\ref{vect2})), and a representation $\rho_{nm}:\mathsf{SL}_{d}(A)\rightarrow \mathsf{SL}_{d}(A)$ such that  for every $g\in \mathsf{SL}_{n}(A)$ with $||g||\geq 2|v_1\cdot v_2|$, the following bound holds \begin{align}\label{rho2nm} |(\rho_{nm}(g)v_1-v_1)\cdot v_2|\geq ||\rho_{nm}(g)||.\end{align} Given the discrete subgroups $\Gamma_1,\Gamma_2$ of $\mathsf{SL}_\ell (k)$, there is  $\varepsilon>0$ with  $$||g-\textup{I}_{\ell}||\geq \varepsilon \ \  \forall \ g\in (\Gamma_1\cup \Gamma_2)\smallsetminus \{\textup{I}_\ell \}.$$ Fix $\nu \in A$ with $|\nu|\geq 1+\frac{2}{\varepsilon}|v_1\cdot v_2|$. By the definition of the representation $\iota_{\nu}$ in (\ref{def-iota}), for every $g\in \Gamma_i\smallsetminus \{\textup{I}_\ell\}$ it follows that $$||\iota_{\nu}(g)||\geq |\nu|\cdot ||g-\textup{I}_\ell ||\geq \max\big\{2|v_1 \cdot v_2|, ||g-\textup{I}_{\ell}||\big\}.$$ Any non-trivial element of the subgroup $\iota_{\nu}(\Gamma_i)$ of $\mathsf{SL}_{n}(A)$ satisfies (\ref{rho2nm}), and $\iota_{\nu}(\Gamma_i)$ is metrically proper (resp. quasi-isometrically embedded) in $\mathsf{SL}_{n}(k)$ if and only if $\Gamma_i$ is. Proposition \ref{theorem-1} applied to $\rho_{nm}(\iota_{\nu}(\Gamma_1))$ and $\rho_{nm}(\iota_{\nu}(\Gamma_2))$ shows that there is $w\in \mathsf{SL}_{d+1}(A)$ such that \begin{align*} \Bigg\langle \begin{pmatrix}[0.9] 1& \\ & \rho_{nm}(\iota_{\nu}(\Gamma_1))  \end{pmatrix}, w\begin{pmatrix}[0.9] 1& \\ & \rho_{nm}(\iota_{\nu}(\Gamma_2))  \end{pmatrix}w^{-1} \Bigg\rangle\end{align*} is a discrete (resp. metrically proper or quasi-isometrically embedded) subgroup of $\mathsf{SL}_{d+1}(k)$, isomorphic to $\Gamma_1\ast \Gamma_2$, and contained in $\mathsf{SL}_{d+1}(A)$.  \end{proof}

\begin{proof}[Proof of Theorem \ref{thm-embd}] Let $m:=n^2$. By applying Lemma \ref{lemma} for $A=k$, there is a representation $\rho_{nm}:\mathsf{SL}_n(k)\rightarrow \mathsf{SL}_d(k)$, $d=\binom{n+m}{n}$, and $v_1,v_2\in k^d$ and $M>1$ such that for any $g\in \mathsf{SL}_n(k)$ with $||g||>M$, \begin{align*} |(\rho_{nm}(g)v_1-v_1)\cdot v_2|\geq ||\rho_{nm}(g)||.\end{align*}

Since $\Gamma$ is discrete in $\mathsf{SL}_n(k)$ and $F:=\big\{g\in \Gamma: ||(\rho_{nm}(g)v_1-v_1)\cdot v_2|\leq ||\rho_{nm}(g)||\big\}$ is contained in $\{g\in \Gamma:||g||\leq M\}$, $F$ is finite. By Baire's category theorem, there are $v_1',v_2'\in k^d$, arbitrarily close to $v_1,v_2$ respectively, with $$|(\rho_{nm}(g)v_1'-v_1')\cdot v_2'|\geq \frac{1}{2} ||\rho_{nm}(g)||$$ for every $g\in (\Gamma \smallsetminus \{\textup{I}_{n}\})\cup (\mathsf{SL}_n(k)\smallsetminus \{h:||h||\leq M\})$. Up to conjugating the representation $$g\mapsto \begin{pmatrix}  1 & \\ & \rho_{nm}(g)\end{pmatrix}, \ g\in \mathsf{SL}_n(k)$$ in $\mathsf{GL}_{d+1}(k)$, we obtain a representation $\tau:\mathsf{SL}_n(k)\rightarrow \mathsf{SL}_{d+1}(k)$ and $0<\epsilon<1$, such that for any $g\in (\Gamma \smallsetminus \{\textup{I}_{n}\})\cup (\mathsf{SL}_n(k)\smallsetminus \{h:||h||\leq M\})$, \begin{align}\nonumber \tau(g)B_{\epsilon}([e_1])&\subset \mathbb{P}(k^{d+1})\smallsetminus \mathcal{N}_{\epsilon}(e_{d+1}^{\perp})\\ \label{Gamma-pp} ||\tau(g)v||&\geq \epsilon ||v||, \ \forall \  [v]\in B_{\epsilon}([e_1]). \end{align} 

Let us set $$T:=\epsilon^{-3}\big(\max \big\{||\tau(h)||: ||h||\leq 6M\big\}\big)^2.$$ By following a standard ping-pong argument (e.g. as in \cite{Tits}), there is a Zariski-dense, discrete finitely generated, free subgroup $\Delta$ of $\mathsf{SL}_{d+1}(k)$ such that for any $g'\in \Delta\smallsetminus \{\textup{I}_{d+1}\}$,
\begin{align}\nonumber g' \Big(\mathbb{P}(k^{d+1})\smallsetminus \mathcal{N}_{\epsilon}\big(e_{d+1}^{\perp}\big)\Big)&\subset B_{\epsilon}([e_1])\\ \label{Delta-pp} ||g'v||\geq T||v|| \ \ \forall \  [v']\in \mathbb{P}&(k^d)\smallsetminus \mathcal{N}_{\epsilon}\big(e_{d+1}^{\perp}\big).\end{align} By the previous inclusions, as in the proof of Proposition \ref{theorem-1}, the subgroup $\langle \tau(\Gamma),\Delta\rangle\simeq \Gamma \ast \Delta$ of $\mathsf{SL}_{d+1}(k)$ is discrete and for any product $b_1\tau(a_1)\cdots \tau(a_{\ell-1}) b_{\ell}\tau(a_{\ell})\in \mathsf{SL}_{d+1}(k)$, where $a_{1},\ldots,a_{\ell-1}\in \tau(\Gamma)\cup (\mathsf{SL}_n(k)\smallsetminus \{h:||h||\leq M\})$ are non-trivial, $b_1,\ldots, b_{\ell}\in \Delta\smallsetminus \{\textup{I}_{d+1}\}$, $a_{\ell}\in \tau(\Gamma)\cup (\mathsf{SL}_n(k)\smallsetminus \{h:||h||\leq M\})$ and $\ell\geq 2$, the previous ping-pong conditions imply  \begin{equation}\label{bound-cor1} ||b_1\tau(a_1)\cdots \tau(a_{\ell-1})b_{\ell}\tau(a_\ell)\big|\big|\geq \sqrt[]{T^{\ell}}||\tau(a_1)||\cdot||b_1||\cdots ||\tau(a_{\ell-1})||\cdot||b_{\ell}||\cdot ||\tau(a_{\ell})||.\end{equation}\enlargethispage{.5\baselineskip} We claim that $\langle \tau(\Gamma),\Delta\rangle \cap \tau(\mathsf{SL}_n(k))=\tau(\Gamma)$. First, for any $g'\in \Delta$ non-trivial, $||(g')^{\pm 1}||\geq T$ and hence by the choice of $T>0$, $$||(g')^{-1}||\geq 2\max\{||\tau(h)||:||h||\leq 6M\}.$$ If $g'\in \tau(\mathsf{SL}_n(k))$, \hbox{there is $h\in \mathsf{SL}_n(k)$ with $||h||\geq 6M$} and $\tau(h)g'=\textup{I}_{d+1}$, which is absurd since the ping-pong inclusion (\ref{Gamma-pp}) and (\ref{Delta-pp}) give $$||\tau(h)g'||\geq \epsilon T>1.$$ This shows $\Delta \cap \tau(\mathsf{SL}_n(k))=\{\textup{I}_{d+1}\}$. 

If $\tau(\mathsf{SL}_n(k))\cap \langle \tau(\Gamma),\Delta\rangle$ were not $\tau(\Gamma)$, there would exist $\ell\geq 2$ and $b_1,\ldots,b_{\ell}\in \Delta$ non-trivial, $\tau(a_1),\ldots,\tau(a_{\ell-1})\in \tau(\Gamma)$ non-trivial such that $g'':=b_1\tau(a_1)\cdots b_{\ell-1}\tau(a_{\ell-1})b_{\ell}\in \tau(\mathsf{SL}_n(k))$. By (\ref{bound-cor1}), $$||(g'')^{-1}||\geq 2\max\{||\tau(g)||:||g||\leq 6M\}$$ hence there is $a_{\ell}\in \mathsf{SL}_n(k)$ with $||a_{\ell}||\geq 6M$ such that $$g''\tau(a_{\ell})=b_1\tau(a_1)\cdots b_{\ell}\tau(a_{\ell})=\textup{I}_{d+1}$$ which is absurd again by (\ref{bound-cor1}). This verifies $\tau(\mathsf{SL}_n(k))\cap \langle \tau(\Gamma),\Delta\rangle=\tau(\Gamma)$ and completes the proof of the theorem. \end{proof} 

Margulis' normal subgroup theorem \cite[Thm. IX 5.6]{Margulis} implies that when $\Gamma$ is not virtually free, $\Gamma'$ above fails to virtually be a lattice in the $k_0$-points of any $k_0$-semisimple group over any local field $k_0$. 

We close this section with two further consequences of Theorem \ref{mainthm-1'} on linearity of free products.

\begin{proof}[Proof of Corollary \ref{R[t]}] Let $\textup{Frac}(\mathsf{R})$ be the field of fractions of $\mathsf{R}$. Consider the field of rational functions $\textup{Frac}(\mathsf{R})(t)$ equipped with the discrete valuation $$\textup{val}\left(f_1f_2^{-1}\right):=\textup{deg}f_2-\textup{deg}f_1,\  \ f_1,f_2\in \mathsf{R}[t], \ f_2\neq 0$$ where $\textup{deg}f_i$ is the degree of $f_i$. Note that $\mathsf{R}[t]$ is discrete in $\textup{Frac}(\mathsf{R})(t)$ and $v(\mathsf{R}[t])=1$, hence the theorem follows immediately by applying Theorem \ref{mainthm-1'} for $A=\mathsf{R}[t]$.
\end{proof}

\begin{corollary}\label{SO} Let $n\geq 2$ be an integer and $d=\binom{4n^2+2n}{2n}$.
\begin{enumerate}\setlength{\itemsep}{1.5pt}

\item \label{SO1} For any prime $p$,  $\mathsf{SO}_n(\mathbb{Z}[\frac{1}{p}])\ast \mathbb{Z}$ embeds as a subgroup of $\mathsf{SL}_{d+1}(\mathbb{Z}[\frac{1}{p}])$.

\item \label{SO2} Let $\mathsf{R}$ be a subdomain of $\mathbb{R}$ and $\mathsf{Frac}(\mathsf{R})$ be the field of fractions of $\mathsf{R}$. For any irreducible polynomial $q(t)\in \mathsf{Frac}(\mathsf{R})[t]$ the free product $\mathsf{SO}_n(\mathsf{R}[t,\frac{1}{q(t)}])\ast \mathbb{Z}$ embeds as a subgroup of $\mathsf{SL}_{d+1}(\mathsf{R}[t,\frac{1}{q(t)}])$.
\end{enumerate}

\end{corollary}

\begin{proof} For (\ref{SO1}), $\mathsf{SO}_n(\mathbb{Z}[\frac{1}{p}])$ is a discrete subgroup of $\mathsf{SL}_n(\mathbb{Z}[\frac{1}{p}])$. Then (\ref{SO1}) follows by applying Theorem \ref{mainthm-1'} for $A=\mathbb{Z}[\frac{1}{p}]$, $\Gamma_1=\mathsf{SO}_n(\mathbb{Z}[\frac{1}{p}])$ and $\Gamma_2=\mathbb{Z}$.

\par For (\ref{SO2}), fix the $q(t)$-adic valuation $\textup{val}:\textup{Frac}(\mathsf{R})(t) \rightarrow \mathbb{Z}\cup\{\infty\}$ where, for $F(t)\in \textup{Frac}(\mathsf{R})[t]$, $\textup{val}(F(t))$ is the largest integer $\ell$ such that $q(t)^\ell$ divides $F(t)$.

Then $\mathsf{SO}_n(\mathsf{R}[t,\frac{1}{q(t)}])$ is a discrete subgroup of $\mathsf{SL}_n(\textup{Frac}(\mathsf{R})(t))$. To see this, any non-trivial element $g=(g_{ij}(t))_{i,j=1}^{n}$ in $\mathsf{SO}_n(\mathsf{R}[t,\frac{1}{q(t)}])$ satisfies $$g_{11}(t)^2+g_{12}(t)^2+\cdots+g_{nn}(t)^2=n.$$ There are two possibilities, either $\min_{ij}\textup{val}(g_{ij}(t))\geq 0$ or $g_{ij}(t)=\frac{p_{ij}(t)}{q(t)^{m_{ij}}}$ for some $(i,j)$, $m_{ij}\geq 1$ and coprime polynomials $p_{ij}(t),q(t)\in \textup{Frac}(\mathsf{R})[t]$. In the first case $g_{ij}(t)\in \textup{Frac}(\mathsf{R})[t]$ for every $(i,j)$, hence $g\in \mathsf{SO}_n(\mathsf{R})$ and $||g-\textup{I}_n||=1$, and in the latter $||g-\textup{I}_n||\geq 1$ holds. 
 
Thus, $\mathsf{SO}_n(\mathsf{R}[t,\frac{1}{q(t)}])$ is  discrete  in $\mathsf{SL}_n(\textup{Frac}(\mathsf{R})(t))$ and (\ref{SO2}) follows by applying Theorem \ref{mainthm-1'} for $A=\mathsf{R}[t,\frac{1}{q(t)}]$, $\Gamma_1=\mathsf{SO}_n(\mathsf{R}[t,\frac{1}{q(t)}])$ and $\Gamma_2=\mathbb{Z}$.
\end{proof}

\section{Nonarchimdean quasi-isometric  amalgams of free groups}\label{Gauss}

\subsection{Gauss valuation and quasi-isometric embedding} In this section, $k$ denotes a nonarchimedean local field equipped with a discrete valuation $\textup{val}:k\rightarrow \mathbb{Z}\cup \{\infty\}$ and let $k(t)$ be the field of rational functions field over $k$ of one variable. The {\em Gauss valuation} is the discrete  valuation $\textup{val}: k(t) \rightarrow \mathbb{Z} \cup \{\infty\}$ extending the valuation on $k$ as follows:
 for $a_0+\dots+a_rt^r \in k[t]$,  $$\textup{val}(a_0+\dots+a_rt^r ) := \min_{0 \leq i \leq r} \textup{val}(a_i)$$ and $\textup{val}(f_1 f_2^{-1}) = \textup{val}(f_1) - \textup{val}(f_2)$ for $f_1, f_2 \in k[t]$ with $f_2 \neq 0$. For more background see \cite[Ch. VI, \S 14]{Zar-Sam}. 

 Equip $\mathsf{SL}_d(k(t))$ with the $\ell_{\infty}$-norm induced by the Gauss valuation on $k(t)$. Clearly any discrete subgroup of $\mathsf{SL}_d(k)$ is metrically proper in $\mathsf{SL}_d(k(t))$. A representation $\rho:\Delta \rightarrow \mathsf{SL}_d(k(t))\times \mathsf{SL}_d(k(t))$ of a finitely generated group $\Delta$ is a quasi-isometric embedding if $\rho$ is a quasi-isometric embedding in $\mathsf{SL}_{2d}(k(t))$.

For the definition of Anosov representations in $\mathsf{SL}_d(k)$, where $k$ is a nonarchimedean local field, we refer to \cite[Def. 5.35]{KLP-nonar} (see also \cite[Rmk. 1.6 (a)]{GGKW}). Fix a Cartan decomposition $\mathsf{SL}_2(k)=\mathsf{K}\mathsf{A}^{+}\mathsf{K}$, where $\mathsf{K}$ is compact and $\mathsf{A}^{+}$ are the diagonal matrices whose diagonal entries have moduli in non-increasing order. The following lemma is a consequence of the dynamical properties of the limit maps of Anosov groups (e.g. see \cite[Prop. 3.5]{PSW}). 

\begin{lemma}\label{doublecoset} Let $\Delta$ be a quasi-isometrically embedded subgroup of $\mathsf{SL}_2(k)$ and $\langle w\rangle$ a maximal cyclic subgroup of $\Delta$ contained in the diagonal subgroup of $\mathsf{SL}_2(k)$. There is $\varepsilon>0$ such that $$\Delta=\langle w\rangle \Delta_{\varepsilon } \langle w\rangle$$where $\Delta_{\varepsilon }:= \bigcap\big(\Delta \cap \big\{g\in \mathsf{SL}_2(k): |g_{jj'}|\geq \varepsilon ||g||,  j,j'=1,2\big\}\big)$.\end{lemma} 

\begin{proof} Let $\partial_{\infty}\Delta$ be the Gromov boundary of $\Delta$ and $\overline{\Delta}:=\Delta \sqcup \partial_{\infty}\Delta$ the bordification of $\Delta$, equipped with a visual metric. Denote by $w^{+}, w^{-}\in \partial_{\infty}\Delta$ the fixed points of $w$. As $\langle w\rangle$ acts cocompactly on $\overline{\Delta}\smallsetminus \{w^{+},w^{-}\}$, there are $\epsilon, \epsilon'>0$ such that \begin{align*}\Delta=\langle w\rangle \mathcal{X}_{\epsilon} & \langle w\rangle\\ \mathcal{X}_{\epsilon}:=\{\gamma \in \Delta:\textup{dist}(\gamma^{\pm 1}, \{w^{+},w^{-1}\})\geq \epsilon \}&\cap\{\gamma \in \Delta: \textup{dist}(\gamma w^{\pm}, \{w^{+},w^{-}\})\geq \epsilon' \}\end{align*} and $\textup{dist}(\star, \{w^{+},w^{-}\})$ denotes the distance of $\star \in \overline{\Delta}$ from $\{w^{+},w^{-}\}$ with respect to the visual metric. In particular, there is $\epsilon_1>0$ such that for every $\gamma \in \mathcal{X}_{\epsilon}$, \begin{align}\label{epsilon1-bound} |\gamma e_i\cdot e_j|\geq \epsilon_1 ||\gamma e_i||, \ i,j=1,2.\end{align}

Since $\Delta$ is quasi-isometrically embedded in $\mathsf{SL}_2(k)$ it is Anosov and let $\xi: \partial_{\infty}\Delta \rightarrow \mathbb{P}(k^2)$ be the Anosov limit map where $\{\xi(w^{+}),\xi(w^{-})\}=\{[e_1], [e_2]\}$ are the eigenlines of $w$. For any infinite sequence $(\gamma_n)$ in $\mathcal{X}_{\epsilon}$, any accumulation point of $(\gamma_n^{\pm 1})$ avoids a fixed neighborhood of $\{w^{+},w^{-}\}$ in $\overline{\Delta}$. By writing $\gamma_n=w_{\gamma_n}a(\gamma_n)w_{\gamma_n}'$ in the Cartan decomposition, up to subsequence, $\xi(\lim\gamma_n^{-1})=\lim [(w_{\gamma_n}')^{-1}e_2]$ and thus both coordinates of any accumulation point of the vectors $((w_{\gamma_n}')^{-1}e_2)$ have absolute value bounded below by a positive constant independent from $(\gamma_n^{\pm 1})$. Therefore, there is a finite subset $F$ of $\mathcal{X}_{\epsilon}$ with $$\epsilon_2:=\inf_{\gamma \in \mathcal{X}_{\epsilon}\smallsetminus F} \Big\{|w_{\gamma}'e_j\cdot e_1|: \gamma=w_{\gamma}a(\gamma)w_{\gamma}', a(\gamma)\in \mathsf{A}^{+}, w_{\gamma},w_{\gamma}'\in \mathsf{K}, j=1,2\Big\}>0.$$ This shows that for every $\gamma \in \mathcal{X}_{\epsilon}\smallsetminus F$, $$||\gamma e_j||\geq |w_{\gamma}'e_j\cdot e_1|\cdot ||\gamma||\geq \epsilon_2||\gamma||.$$ For any element of $\mathcal{X}_{\epsilon}$ its entries are non-zero, thus, by (\ref{epsilon1-bound}) and since $\epsilon_2>0$, there is $\varepsilon>0$ such that $\min_{ij}|\gamma_{ij}|\geq \varepsilon ||\gamma||$ for any $\gamma \in \mathcal{X}_{\epsilon}$. This completes the proof of the lemma. \end{proof}

The main result of this section is the following theorem from which Theorem \ref{qie-thm} follows.

\begin{theorem}\label{prod-qie} Let $k$ be a local field and $\Gamma_1,\Gamma_2$ be quasi-isometrically embedded subgroups of $\mathsf{SL}_2(k)$ such that $\Gamma_1\cap \Gamma_2$ is a maximal cyclic subgroup of $\Gamma_i$, $i=1,2$. 
\begin{enumerate}
\item \label{part-1} Set $V:=\textup{Sym}^4k(t)^2\simeq k(t)^5$ and assume $k$ is of chararcteristic not equal to $2$ or $3$.  There exists $\alpha\in \mathsf{GL}(V)$, centralizing the group $\textup{sym}^4(\Gamma_1\cap \Gamma_2)$, such that the representation $\rho:\Gamma_1\ast_{\Gamma_1\cap \Gamma_2} \Gamma_2\rightarrow \mathsf{SL}(V)\times \mathsf{SL}(V)$, $$\rho(\gamma)= \left\{\begin{matrix}[1.2]
\big(\textup{sym}^4(\gamma), \textup{sym}^4(\gamma) \big)& \gamma\in \Gamma_1  \\
\big(\alpha \textup{sym}^4(\gamma) \alpha^{-1}, \alpha^{-1}\textup{sym}^4(\gamma)\alpha \big)  & \gamma\in \Gamma_2  \\
\end{matrix}\right.$$ is a quasi-isometric embedding.

\item \label{part-2} There exists $\beta\in \mathsf{SL}_2(k(t))$, centralizing the group $\Gamma_1\cap \Gamma_2$, such that the representation~$\rho':\Gamma_1\ast_{\Gamma_1\cap \Gamma_2} \Gamma_2\rightarrow \mathsf{SL}_2(k(t))\times \mathsf{SL}_2(k(t))$, $$\rho'(\gamma)= \left\{\begin{matrix}[1.2]
(\gamma, \gamma)& \gamma\in \Gamma_1  \\
(\beta \gamma \beta^{-1}, \beta^{-1}\gamma\beta )  & \gamma\in \Gamma_2  \\
\end{matrix}\right.$$ is faithful and has metrically proper image.\end{enumerate}
\end{theorem}

\begin{proof} By conjugating $\Gamma_1$ and $\Gamma_2$ with an element of $\mathsf{GL}_2(k)$, $\Gamma_1\cap \Gamma_2=\langle w\rangle$ and  $w=\textup{diag}(\pi,\pi^{-1})$ for some $\pi \in k$ with $|\pi|>1$. Fix a metric on $\Gamma_i$ and let $|\cdot|_i:\Gamma_i\rightarrow \mathbb{N}$ the be associated word length function. Since $\Gamma_i$ is quasi-isometrically embedded in $\mathsf{SL}_2(k)$, there are $c,a>0$ such that \begin{align}\label{qie}||\gamma||\geq e^{a|\gamma|_i-c}, \ \forall \ \gamma \in \Gamma_i. \end{align}  By Lemma \ref{doublecoset}, there is $\varepsilon>0$, such that the following \hbox{decompositions hold} \begin{align}\label{decomp-Gamma}\Gamma_i&=(\Gamma_1\cap \Gamma_2)\Gamma_{i\varepsilon}(\Gamma_1\cap \Gamma_2)\end{align} where $\Gamma_{i\varepsilon}:= \bigcap\big(\Gamma_i \cap \big\{g\in \mathsf{SL}_2(k): |g_{jj'}|\geq \varepsilon ||g||,  j,j'=1,2\big\}\big)$. In order to simplify the notation, set $\Gamma:=\Gamma_1\ast_{\langle w\rangle}\Gamma_2$; the latter group is Gromov-hyperbolic by \cite{BF92}.

\begin{proof}[Proof of Theorem \ref{prod-qie}(\ref{part-1}).] Assume that $\textup{char}k\neq 2,3$.  Define the subsets of $\Gamma$ as follows: \begin{align*} \mathcal{A}_1&:=\big\{\gamma_1 (w^{r_1}\gamma_2w^{s_1})\cdots \gamma_{2\ell-1} (w^{r_\ell}\gamma_{2\ell}w^{s_\ell}): \gamma_{2i-1} \in \Gamma_{1\varepsilon}, \gamma_{2i}\in \Gamma_{2\varepsilon}, r_i,s_i\in \mathbb{Z}, \ell \in \mathbb{N} \big\}\\ \mathcal{A}_2&:=\big\{(w^{r_1}\gamma_1w^{s_1})\gamma_2 \cdots (w^{r_\ell}\gamma_{2\ell-1}w^{s_\ell})\gamma_{2\ell}: \gamma_{2i-1} \in \Gamma_{2\varepsilon}, \gamma_{2i}\in \Gamma_{1\varepsilon}, r_i,s_i\in \mathbb{Z}, \ell\in \mathbb{N} \big\}.\end{align*}

Fix the ordered basis $(e_1^2e_2^2,e_1^4,e_2^4, e_1^3e_2,e_2^3e_1)$ of $V=\textup{Sym}^4k(t)^2$ with respect to which the symmetric power $\textup{sym}^4 :\mathsf{SL}_2(k(t))\rightarrow \mathsf{SL}(V)$ has the following matrix form \begin{align}\label{sym4} \textup{sym}^4 \left(\begin{pmatrix}\kappa  & \lambda \\ \mu & \nu \end{pmatrix}\right)=\begin{pmatrix}[1.05]
1+6\kappa \lambda \mu \nu & 6\kappa^2\mu^2  & 6\lambda^2\nu^2  &  \cdots \\
\kappa^2\lambda^2 & \kappa^4  & \lambda^4  & \cdots  \\
 \mu^2 \nu^2& \mu^4 & \nu^4 & \cdots  \\
 \vdots &  \vdots &  \vdots & \ddots \\
\end{pmatrix}\end{align} where the other entries of the matrix above are polynomials in the variables $\{\kappa,\lambda, \mu,\nu\}$.

 Fix $z_1,z_2\in k$ such that $|z_i|\geq \frac{e^{8c}}{|6|\varepsilon^{18}}$ and $\textup{val}\left(z_1z_2^{-1}\right)=1$ and set $$\alpha:= \begin{pmatrix}t^{-1} & & \\ &  tz_1 & &   \\ & & tz_2 &   \\ & & & \textup{I}_2 \end{pmatrix}.$$ Define the representations $\rho_{\pm}: \Gamma \rightarrow \mathsf{SL}(V)$ as follows, $$\rho_{\pm}(\gamma)= \left\{\begin{matrix}[1.1]
\textup{sym}^4(\gamma) & \gamma\in \Gamma_1  \\
\alpha^{\pm 1} \textup{sym}^4(\gamma) \alpha^{\mp 1}  & \gamma\in \Gamma_2  \\
\end{matrix}.\right.$$ Note that $\rho_{\pm}$ is well-defined since $\alpha$ commutes with  $\textup{sym}^4(w)= \textup{diag}(1,\pi^4,\pi^{-4}, \pi^2,\pi^{-2})$. Fix two words of the form \begin{align*} \gamma_1 (w^{r_1}\gamma_2w^{s_1})\cdots \gamma_{2\ell-1} (w^{r_\ell}\gamma_{2\ell}w^{s_\ell})&\in \mathcal{A}_1, \ \gamma_{2i-1} \in \Gamma_{1\varepsilon}, \gamma_{2i}\in \Gamma_{2\varepsilon}, r_i,s_i\in \mathbb{Z}\\ (w^{r_1'}\gamma_1'w^{s_1'})\gamma_2' \cdots (w^{r_\ell'}\gamma_{2\ell-1}'w^{s_\ell'})\gamma_{2\ell}'&\in \mathcal{A}_2, \ \gamma_{2i-1}' \in \Gamma_{2\varepsilon}, \gamma_{2i}'\in \Gamma_{1\varepsilon}, r_i',s_i'\in \mathbb{Z}. \end{align*} By using (\ref{sym4}), for every $i=1,\ldots, \ell$, the products \begin{align*} \rho_{+}(\gamma_{2i-1}w^{r_i}\gamma_{2i}w^{s_i})&=\textup{sym}^4(\gamma_{2i-1})\alpha \big( \textup{sym}^4 (w^{r_i}\gamma_{2i}w^{s_i})\big)\alpha^{-1}\\ \rho_{-}(w^{r_i'}\gamma_{2i-1}'w^{s_i'}\gamma_{2i}')&=\alpha^{-1}\big(\textup{sym}^4(w^{r_i'}\gamma_{2i-1}'w^{s_i'})\big)\alpha \big(\textup{sym}^4 (\gamma_{2i}')\big) \end{align*} have their $(1,1)$-entry a polynomial in $k[t,t^{-1}]$ of degree $2$ such that the coefficients of $t^2$ respectively are  \begin{align*}D_i:&= 6z_1 \Big( (\gamma_{2i-1})_{11}(\gamma_{2i-1})_{21} (w^{r_i}\gamma_{2i}w^{s_i})_{11} (w^{r_i}\gamma_{2i}w^{s_i})_{12} \Big)^2+\\ &+6z_2 \Big((\gamma_{2i-1})_{12}(\gamma_{2i-1})_{22}(w^{r_i}\gamma_{2i}w^{s_i})_{21} (w^{r_i}\gamma_{2i}w^{s_i})_{22}\Big)^2\\
&=6z_1 \pi^{4r_i} \Big( (\gamma_{2i-1})_{11}(\gamma_{2i-1})_{21} (\gamma_{2i})_{11} (\gamma_{2i})_{12} \Big)^2+\frac{6z_2}{\pi^{4r_i}} \Big((\gamma_{2i-1})_{12}(\gamma_{2i-1})_{22}(\gamma_{2i})_{21} (\gamma_{2i})_{22}\Big)^2\\ D_i':&= 6z_1 \Big( (w^{r_i'}\gamma_{2i-1}'w^{s_i'})_{11}(w^{r_i'}\gamma_{2i-1}'w^{s_i'})_{21} (\gamma_{2i}')_{11} (\gamma_{2i}')_{12} \Big)^2+\\ &+6z_2 \Big((w^{r_i'}\gamma_{2i-1}'w^{s_i'})_{12}(w^{r_i'}\gamma_{2i-1}'w^{s_i'})_{22}(\gamma_{2i}')_{21} (\gamma_{2i}')_{22}\Big)^2 \\ 
&=6z_1 \pi^{4s_i'} \Big( (\gamma_{2i-1}')_{11}(\gamma_{2i-1}')_{21} (\gamma_{2i}')_{11} (\gamma_{2i}')_{12} \Big)^2+\frac{6z_2}{\pi^{4s_i'}} \Big((\gamma_{2i-1}')_{12}(\gamma_{2i-1}')_{22}(\gamma_{2i}')_{21} (\gamma_{2i}')_{22}\Big)^2.  \end{align*} 

By (\ref{decomp-Gamma}),  for every $j,j'=1,2$ and $\bigast=2i-1,2i$,  \begin{align*}|(\gamma_{ \bigast})_{jj'}|\geq \varepsilon ||\gamma_{ \bigast}||, \ |(\gamma_{ \bigast }')_{jj'}|\geq \varepsilon ||\gamma_{ \bigast }'||. \end{align*} Moreover, since $\textup{val}(z_1z_2^{-1})=1$,  by the ultrametric property\footnote{For $x,y\in k^{\ast}$, $\textup{val}(x^{-2}y^2)\in 2\mathbb{N}$. As $\textup{val}(z_1z_2^{-1})=1$, $|z_1x^2|\neq |z_2y^2|$, $|z_1x^2+z_2y^2|=\max\{|z_1x^2|,|z_2y^2|\}$.} of the absolute value on $k$ and since $|z_1|, |z_2|\geq \frac{e^{8c}}{|6|\varepsilon^{18}}$, we obtain the bounds \begin{align*}|D_i|=  |6|\max\Big\{|z_1 \pi^{4r_i}| &\big|(\gamma_{2i-1})_{11}(\gamma_{2i-1})_{21} (\gamma_{2i})_{11} (\gamma_{2i})_{12}\big|^2,|z_2\pi^{-4r_i}|\big|(\gamma_{2i-1})_{12}(\gamma_{2i-1})_{22}(\gamma_{2i})_{21} (\gamma_{2i})_{22}\big|^2\Big\}\\ & \geq e^{8c} |\pi|^{4|r_i|} \big(||\gamma_{2i-1}||\cdot ||\gamma_{2i}||\big)^4\\ |D_i'|= |6|\max\Big\{|z_1 \pi^{4s_i'}| &\big|(\gamma_{2i-1}')_{11}(\gamma_{2i-1}')_{21}  (\gamma_{2i}')_{11} (\gamma_{2i}')_{12}\big|^2,|z_2\pi^{-4s_i'}|\big|(\gamma_{2i-1}')_{12}(\gamma_{2i-1}')_{22}(\gamma_{2i}')_{21} (\gamma_{2i}')_{22}\big|^2\Big\}\\ & \geq e^{8c} |\pi|^{4|s_i'|}\big( ||\gamma_{2i-1}'||\cdot ||\gamma_{2i}'||\big)^4. \end{align*} 

 Each $(j,1)$-entry (resp. $(1,j)$-entry) of $\rho_{+}(\gamma_{2i-1}w^{r_i}\gamma_{2i}w^{s_i})$ (resp. $\rho_{-}(w^{r_i'}\gamma_{2i-1}'w^{s_i'}\gamma_{2i}')$)  is a polynomial in $k[t,t^{-1}]$ of degree at most $2$ while all others entries are of degree at most  $1$. The $(1,1)$-entry of $\rho_{+}\big(\gamma_1 (w^{r_1}\gamma_2w^{s_1})\cdots \gamma_{2\ell-1} (w^{r_\ell}\gamma_{2\ell}w^{s_\ell})\big)$ and $\rho_{-}\big(\gamma_1' (w^{r_1'}\gamma_2'w^{s_1'})\cdots \gamma_{2\ell-1}'(w^{r_\ell'}\gamma_{2\ell}'w^{s_\ell'})\big)$ is a polynomial in $k[t,t^{-1}]$  whose leading terms are $(D_1\cdots D_{\ell}) t^{2\ell}$ and $(D_1'\cdots D_{\ell}') t^{2\ell}$ respectively. Therefore, their Gauss norms satisfy the estimates \begin{align*}\big|\big|\rho_{+}\big(\gamma_1 (w^{r_1}\gamma_2w^{s_1})\cdots \gamma_{2\ell-1} (w^{r_\ell}\gamma_{2\ell}w^{s_\ell})\big)\big|\big|&\geq |D_1\cdots D_{\ell}|\\  \big|\big|\rho_{-}\big(\gamma_1' (w^{r_1'}\gamma_2'w^{s_1'})\cdots \gamma_{2\ell-1}' (w^{r_\ell'}\gamma_{2\ell}'w^{s_\ell'})\big)\big|\big|&\geq   |D_1'\cdots D_{\ell}'|.\end{align*} By using the previous bounds for $|D_i|$ and $|D_i'|$, we conclude that \begin{align}\label{rho+}\big|\big|\rho_{+}\big(\gamma_1 (w^{r_1}\gamma_2w^{s_1})\cdots \gamma_{2\ell-1} (w^{r_\ell}\gamma_{2\ell}w^{s_\ell})\big)\big|\big|\geq  e^{8\ell c} |\pi|^{4\sum |r_i|}\big(||\gamma_{1}||\cdots  ||\gamma_{2\ell}||\big)^4\\ 
\label{rho-}\big|\big|\rho_{-}\big(\gamma_1' (w^{r_1'}\gamma_2'w^{s_1'})\cdots \gamma_{2\ell-1}' (w^{r_\ell'}\gamma_{2\ell}'w^{s_\ell'})\big)\big|\big|\geq e^{8\ell c}|\pi|^{4\sum |s_i'|}\big(||\gamma_{1}'||\cdots  ||\gamma_{2\ell}'||\big)^4.\end{align}

Since $\Gamma_1,\Gamma_2$ generate $\Gamma$, consider the length function $|\cdot|_{\Gamma}$ induced by a choice of metrics on $\Gamma_{i}$. Fix a finite subset\footnote{Take $\mathcal{F}=\{1,w_{11},w_{12},w_{21},w_{22}\}$ such that $w_{1i},w_{2i}\in \Gamma_{i\varepsilon }$ generate a rank-2 free subgroup of $\Gamma_i$.} $\mathcal{F}$ of $\Gamma_{1\varepsilon}\cup\Gamma_{2\varepsilon}\cup\{1\}$ such that for any $g\in \Gamma$ there are $h_1,h_2\in \mathcal{F}$ (resp. $h_1',h_2'\in \mathcal{F}$) such that $h_1g h_2\in \mathcal{A}_1$ (resp. $h_1'g h_2'\in \mathcal{A}_2$). 

Let $g\in \Gamma$ be non-trivial. By using (\ref{decomp-Gamma}) there are  $h_0,h_0'\in \mathcal{F}$ and $\ell\in \mathbb{N}$ such that $$h_0gh_0'=w^{r_0}\gamma_{1}(w^{r_1}\gamma_2 w^{s_1})\cdots \gamma_{2\ell-1}(w^{r_{\ell}}\gamma_{2\ell}w^{s_{\ell}})$$ where $r_i,s_i\in \mathbb{Z}$, $\gamma_{2i-1}\in \Gamma_{1\varepsilon }, \gamma_{2i} \in \Gamma_{2\varepsilon}$. There are two cases to consider:
\medskip

\noindent {\em Case 1. $r_0=0$.} Choose $h_i\in \mathcal{F}\cap \Gamma_{i\varepsilon }, i=1,2$, such that $h_2(h_0gh_0')h_1\in \mathcal{A}_2$. Since $h_0gh_0'\in \mathcal{A}_1$, by using (\ref{rho+}) and (\ref{rho-}), it follows that  \begin{align*} || \rho_{+}(h_0gh_0')||&\geq e^{8\ell c}|\pi|^{4\sum |r_i|} \big(||\gamma_{1}||\cdots  ||\gamma_{2\ell}||\big)^4\\ || \rho_{-}(h_2(h_0gh_0')h_1)||&\geq e^{8\ell c}|\pi|^{4\sum |s_i|} ||h_1||\cdot ||h_2||\cdot \big(||\gamma_{1}||\cdots  ||\gamma_{2\ell}||\big)^4.\end{align*}

Set $C:=\max\{||h||^4: h\in \mathcal{F}\}$. By the sub-multiplicativity of the operator norm, \begin{align*}C\Big(||\rho_{+}(h_0 gh_0')\big|\big|+\big|\big|\rho_{-}(h_0gh_0')||\Big)&\geq ||\rho_{+}(h_0gh_0')||+||\rho_{-}(h_2(h_0gh_0')h_1)||\\ &\geq   \sqrt[]{4||h_1||\cdot ||h_2||} e^{8\ell c} |\pi|^{2\sum(|r_i|+|s_i|)}\big(||\gamma_{1}||\cdots  ||\gamma_{2\ell}||\big)^4\\ &\geq 2 e^{8\ell c} \big(||\gamma_{1}||\cdot ||w^{r_1}\gamma_{2}w^{s_1}||\cdots ||\gamma_{2\ell-1}||\cdot ||w^{r_\ell}\gamma_{2\ell}w^{s_\ell}||\big)^2\\ &\geq 2e^{4\ell c} e^{2a\big(|\gamma_{1}|_{1}+|w^{r_1}\gamma_{2}w^{s_1}|_2+\cdots+ |\gamma_{2\ell-1}|_{1}+|w^{r_\ell}\gamma_{2\ell}w^{s_\ell}|_2\big)} \\ & \geq e^{2a|h_0 gh_0'|_{\Gamma}}. \end{align*} 

\noindent {\em Case 2. $r_0\neq 0$.} Choose $h_i'\in \mathcal{F}\cap \Gamma_{i\varepsilon}, i=1,2$, such that $h_1'h_2'(h_0gh_0')\in \mathcal{A}_1$. By using the bound from Case 1 above, we obtain \begin{align*}C\big(||\rho_{+}(h_1'h_2'(h_0gh_0'))||+||\rho_{-}(h_1'h_2'(h_0gh_0'))||\big) & \geq e^{2a|h_1' h_2'(h_0 gh_0')|_{\Gamma}}.\end{align*}
It follows that $\rho_{+}\times \rho_{-}:\Gamma \rightarrow \mathsf{SL}_5(k(t))\times \mathsf{SL}_5(k(t))$ is a quasi-isometric embedding and Theorem \ref{prod-qie}(\ref{part-1}) follows.\end{proof}

\begin{proof}[Proof of Theorem \ref{prod-qie}(\ref{part-2})] Recall that $\Gamma=\Gamma_1\ast_{\langle w\rangle}\Gamma_2$. Fix $z\in k$ with $|z|\geq \frac{1}{\varepsilon^4}$ and set $$\beta:=\begin{pmatrix} \frac{1}{tz} & \\ & tz \end{pmatrix}$$ centralizing $\langle w\rangle$. Consider the unique representation $\rho_{\pm}':\Gamma \rightarrow \mathsf{SL}_2(k(t))$ satisfying $$\rho_{\pm}'(\gamma)= \left\{\begin{matrix}
\gamma & \gamma\in \Gamma_1  \\
\beta^{\pm 1} \gamma \beta^{\mp 1}  & \gamma\in \Gamma_2  \\
\end{matrix}\right..$$ 

Fix $a_{1},\ldots,a_{\ell}\in \Gamma_1, b_{0}\ldots,b_{\ell}\in \Gamma_2$ non-trivial. By looking at the matrices $\rho_{+}'(a_1b_1\cdots a_\ell b_\ell)\\=a_1\beta b_1\beta^{-1}\cdots a_\ell \beta b_\ell \beta^{-1}$ (resp. $\rho_{-}'(a_1b_1\cdots a_\ell b_\ell)=a_1\beta^{-1} b_1\beta \cdots a_\ell \beta^{-1} b_\ell \beta$) we check that the $(1,1)$-entry (resp. $(2,2)$-entry)  is a polynomial of degree $2\ell$ whose leading term is $(tz)^{2\ell}(a_{1})_{12}(b_1)_{21}\cdots (a_{\ell})_{12}(b_\ell)_{21}$ (resp. $(tz)^{2\ell}(a_{1})_{21}(b_1)_{12}\cdots (a_{\ell})_{21}(b_\ell )_{12}$). The Gauss norm of the previous products satisfy the bounds \begin{align*}\big|\big|\rho_{+}'(a_1b_1\cdots a_\ell b_\ell)\big|\big|&\geq |z|^{2\ell} \big|(a_{1})_{12}(b_1)_{21}\cdots (a_{\ell})_{12}(b_\ell)_{21}\big|\\ \big|\big|\rho_{-}'(a_1b_1\cdots a_\ell b_\ell)\big|\big|&\geq |z|^{2\ell} \big|(a_{1})_{21}(b_1)_{12}\cdots (a_{\ell})_{21}(b_\ell)_{12}\big|.\end{align*} In particular, by the arithmetic-geometric mean inequality we obtain the bounds \begin{align}\label{AM-GM-bound}\big|\big|\rho_{+}'(a_1b_1\cdots a_\ell b_\ell)\big|\big|+ \big|\big|\rho_{-}'(a_1b_1\cdots a_\ell b_\ell)\big|\big|&\geq  |z|^{2\ell}\sqrt[]{|(a_{1})_{12} (a_1)_{21}\cdots (b_{\ell})_{12}(b_\ell)_{21}|}\end{align} and similarly   \begin{align}\label{AM-GM-bound-2}\hspace{-.1cm} \big|\big|\rho_{+}'(b_0a_1\cdots b_\ell)\big|\big|+ \big|\big|\rho_{-}'(b_0a_1\cdots b_\ell)\big|\big|&\geq |z|^{2\ell} \sqrt[]{|(b_0)_{12}(b_0)_{21}(a_{1})_{12} (a_1)_{21}\cdots (b_{\ell})_{12}(b_\ell)_{21}|}.\end{align} 

The image of $\rho_{+}'\times \rho_{-}':\Gamma \rightarrow \mathsf{SL}_2(k(t))\times \mathsf{SL}_2(k(t))$ is metrically proper. It suffices to prove that any sequence $(g_n)$ in $\Gamma$ satisfying $$\sup_{n\in \mathbb{N}} \big(||\rho_{+}'(g_n)||+||\rho_{-}'(g_n)||\big)\leq M$$ for some $M>1$ is finite.  Up to multiplying $g_n$ from the right and left with an element in a fixed finite subset of $\Gamma_1\cup\Gamma_2$ and using (\ref{decomp-Gamma}), we may write $$g_n=(w^{s_{1n}}a_{1n}')(w^{s_{2n}}b_{1n}')\cdots (w^{s_{(2r-1)n}}a_{rn}')(w^{s_{2rn}}b_{rn}' w^{\ell_n})$$ in reduced form, where $r=r(n)\in \mathbb{N}$, $a_{in}'\in \Gamma_{1\varepsilon}$, $b_{in}'\in \Gamma_{2\varepsilon}$.
Since $(w^{m_1}gw^{m_2})_{12}(w^{m_1}gw^{m_2})_{21}=g_{12}g_{21}$ for any $m_1,m_2\in \mathbb{Z}$ and $g\in \Gamma_{1\varepsilon}\cup \Gamma_{2\varepsilon}$ we have \begin{align*} \big|(w^{s_{(2i-1)n}}a_{in}')_{12}(w^{s_{(2i-1)n}}a_{in}')_{21}\big|&\geq (\varepsilon ||a_{in}'||)^2\\ \big|(w^{s_{2in}}b_{in}')_{12}(w^{s_{2in}}b_{in}')_{21}\big|&\geq (\varepsilon ||b_{in}'||)^2\\ |(w^{s_{2rn}}b_{rn}' w^{\ell_n})_{12}(w^{s_{2rn}}b_{rn}' w^{\ell_n})_{21}|&\geq (\varepsilon ||b_{rn}'||)^2\end{align*}  By (\ref{AM-GM-bound}) we have that $$M\geq ||\rho_{+}'(g_n)||+||\rho_{-}'(g_n)||\geq (\varepsilon |z|)^r ||a_{1n}'||\cdot ||b_{1n}'||\cdots ||a_{rn}'||\cdot   ||b_{rn}'||\geq (\varepsilon |z|)^r.$$ This shows $r\leq \frac{M}{\log\varepsilon |z|}$ and $||a_{in}'||\leq M$, $||b_{in}'||\leq M$ for every $i$. As $\Gamma_i$ are discrete in $\mathsf{SL}_2(k)$, $(b_{in}'),(a_{in}')$ are finite sequences. If $(s_{1n})$ is infinite, for fixed $m$ and infinitely many $n$, $(a_{1m}')^{-1}w^{s_{1n}-s_{1m}}a_{1n}\neq 1$. Since $||\rho_{\pm}(g_{m}^{-1})||=||\rho_{\pm}(g_m)||$, $$||\rho_{+}(g_m^{-1}g_n)||+||\rho_{-}(g_m^{-1}g_n)||\leq M^2.$$  By writing $$g_{m}^{-1}g_n=\underbrace{w^{-\ell_m} (b_{rm}')^{-1}w^{-s_{2rm}}}_{\in \Gamma_2\smallsetminus \{1\}}\cdots \underbrace{(b_{1m}')^{-1}w^{-s_{2m}}}_{\in \Gamma_2\smallsetminus \{1\}} \underbrace{(a_{1m}')^{-1}w^{s_{1n}-s_{1m}}a_{1n}'}_{\in \Gamma_1\smallsetminus \{1\}}\cdots \underbrace{w^{s_{2rn}}b_{rn}'w^{\ell_n}}_{\in \Gamma_2\smallsetminus \{1\}}$$ in reduced form and using (\ref{AM-GM-bound-2}) it follows that $$M^2\geq ||\rho_{+}(g_m^{-1}g_n)||+||\rho_{-}(g_m^{-1}g_n)||\geq (\varepsilon |z|)^{2(r-1)} ||(a_{1m}')^{-1}w^{s_{1n}-s_{1m}}a_{1n}'||.$$ As $w=\textup{diag}(\lambda,\lambda^{-1})$, by the previous bound, the $(1,2)$-entry of $(a_{1m}')^{-1}w^{s_{1n}-s_{1m}}a_{1n}'$ satisfies $$\big|\pi^{s_{1n}-s_{1m}}(a_{1m}')_{22}(a_{1n}')_{12}-\pi^{s_{1m}-s_{1n}}(a_{1m}')_{12}(a_{1n}')_{22}\big|\leq M^2.$$ Since $a_{1n}',a_{1m}'$ have all of their entries non-zero lying in finite subset of $k$, the sequence $(|s_{1n}-s_{1m}|)$ is also finite. This shows that $(w^{s_{1n}}a_{1n}')$ is a finite sequence. 
By continuing similarly, as $r\leq \frac{M}{\log\varepsilon |z|}$ and using (\ref{AM-GM-bound}) and (\ref{AM-GM-bound-2}), $(w^{s_{(2i-1)r}}a_{in}'), (w^{s_{(2i)r}}b_{2in}')$, $(\ell_n)$ are finite sequences and thus $(g_n)$ is finite. It follows that $(\rho_{+}\times \rho_{-})(\Gamma)\simeq \Gamma$ is a metrically proper subgroup of $\mathsf{SL}_2(k(t))\times \mathsf{SL}_2(k(t))$.
\end{proof}\renewcommand{\qedsymbol}{}\end{proof}

\vspace{-0.7cm}

The method of conjugating one amalgam factor with a transcendetal matrix, centralizing the common subgroup, was first used in \cite{Shalen} to prove linearity of amalgams (see also \cite{TT-anosov, DT-cyclic} for certain amalgams along cyclic groups).

\begin{proof}[Proof of Theorem \ref{qie-thm}] Let $\Gamma_1$ be a free group of finite-rank and $\langle w\rangle$ a maximal cyclic subgroup of $\Gamma_1$. Since $\mathsf{SL}_2(k)$ contains a quasi-isometrically embedded free subgroup, there is a quasi-isometric embedding of $\Gamma_1$ in $\mathsf{SL}_2(k)$ such that $\langle w\rangle$ is diagonal.  Fix $\sigma\in k$ transcendental over the subfield of $k$ generated by the entries of the elements of $\Gamma_1$ and let $$\Gamma_2:=\begin{pmatrix} \sigma &\\ & \frac{1}{\sigma}\end{pmatrix}\Gamma_1\begin{pmatrix} \frac{1}{\sigma} &\\ & \sigma\end{pmatrix}.$$ Lemma \ref{doublecoset} implies $\Gamma_1\cap \Gamma_2=\langle w\rangle$. By Theorem \ref{prod-qie}, the amalgam $\Gamma_1\ast_{\langle w\rangle}\Gamma_2$ (which is isomorphic to the double $\Gamma_1\ast_{\langle w\rangle}\Gamma_1$) embeds in $\mathsf{SL}_5(k(t))\times \mathsf{SL}_5(k(t))$ as a quasi-isometrically embedded subgroup. \end{proof}

\begin{remarks} \normalfont{(1) When $\Gamma_1=\Gamma_2$ the double $\Gamma:=\Gamma_1\ast_{\langle w\rangle}\Gamma_2$ in Theorem \ref{prod-qie} contains infinite index surface groups \cite{Gordon-Wilton}. When $\Gamma$ is not even free, $\rho'_{+}\times \rho'_{-}:\Gamma \rightarrow \mathsf{SL}_2(k(t))\times \mathsf{SL}_2(k(t))$ in Theorem \ref{prod-qie}(\ref{part-2}) projects on each factor to a non-metrically proper subgroup of $\mathsf{SL}_2(k(t))$, as $\rho'_{\pm}(\Gamma)$ cannot act properly on the Bruhat--Tits tree $\mathcal{T}$ of $\mathsf{SL}_2(k(t))$ \cite[Thm. 4, p. 21]{Serre}.
\vspace{.12cm}

\noindent (2) In the notation above, $\Gamma$ is cubulated \cite{Hsu-Wise}, virtually compact special \cite{Agol} and embeds as a quasi-isometrically embedded subgroup of $\mathsf{SL}_d(\mathbb{R})$ \cite{DFWZ}, where $d$ depends on the cubulation of $\Gamma$.  Over the field $k(t)$, Theorem \ref{prod-qie}(\ref{part-1}) provides quasi-isometric embeddings of any such double $\Gamma$ in dimension $10$. Theorem \ref{prod-qie}(\ref{part-2}) also implies that $\Gamma$ is linear over $\mathbb{F}_p(x,t)$ for any prime $p$.  The fact that closed surface groups are $\mathbb{F}_p(x,t)$-linear was already established in \cite[Thm. 4]{FLSS} for $p\geq 5$ and in \cite[Thm. 8.4 \& Cor. 5.5]{Button} for any $p$.}\end{remarks}

We close this section with a metrically proper free subgroup of $\mathsf{SL}_2(k_1)\times \mathsf{SL}_2(k_2)$ with indiscrete projection \hbox{on each factor, where $k_1$ and $k_2$ are discretely valued non-local fields.}

\begin{exmp}\label{irreducible}\normalfont{{\em Discrete free subgroups of products projecting indiscretely on each factor.}
 Equip $\mathbb{Q}_p(t)$ with the Gauss norm extending the standard $p$-adic absolute value $|\cdot|=p^{-\textup{val}(\cdot)}$ on $\mathbb{Q}_p$ and let $||\cdot||_p$ be the induced $\ell_{\infty}$-norm on $\mathsf{SL}_2(\mathbb{Q}_p(t))$. Equip  $G:=\mathsf{SL}_2(\mathbb{Q}_5(t))\times \mathsf{SL}_2(\mathbb{Q}_{13}(t))$ with the product $||(\cdot,\cdot)||:=||\cdot||_5+||\cdot||_{13}$. 

Then $G$ contains a metrically proper free subgroup (with respect to $||(\cdot,\cdot)||$) whose projection on each factor is not discrete. Let us set $$\alpha:=\begin{pmatrix}[1.2]\frac{4}{5} & -\frac{3}{5} \\ \frac{3}{5} & \frac{4}{5} \end{pmatrix}, \ \beta:=\begin{pmatrix}[1.2]\frac{12}{13} & -\frac{5}{13} \\ \frac{5}{13} & \frac{12}{13} \end{pmatrix}.$$ Consider the representation $\varphi:\langle \alpha \rangle \ast \langle \beta \rangle \rightarrow \mathsf{SL}_2(\mathbb{Q}(t))$ defined as follows: $$\varphi(\alpha)=\alpha, \ \varphi(\beta)=\begin{pmatrix}\frac{65}{t} & \\  & 1 \end{pmatrix}\beta \begin{pmatrix}\frac{t}{65} & \\  & 1 \end{pmatrix}$$ and the diagonal embedding $\varphi': \langle \alpha \rangle \ast \langle \beta \rangle \rightarrow  G$, $\varphi'(g)=(\varphi(g),\varphi(g))$.

As any nonzero power of $\alpha$ and $\beta$ has its $(1,2)$ and $(2,1)$ entries opposite, for any reduced word $\gamma_1\cdots \gamma_r\in \langle \alpha\rangle \ast \langle \beta \rangle$ one of the entries of $\varphi(\gamma_1\cdots \gamma_r)$ is a polynomial in $\mathbb{Q}[t,t^{-1}]$ of degree $r$ or $r-1$, \hbox{whose leading term coefficient is} $$\pm 65^{\star}(\gamma_1)_{12}\cdots (\gamma_r)_{12}$$ where $\star \in \{-\frac{r}{2},\frac{\pm 1-r}{2}\}$.  For any $z\in \mathbb{Z}[\frac{1}{5},\frac{1}{13}]$ non-zero, $|z|_{\mathbb{E}} |z|_{5} |z|_{13}\geq 1$ (where $|\cdot|_{5,13}$ and $|\cdot|_{\mathbb{E}}$ are the $\{5,13\}$-value and Euclidean value respectively), thus by the arithmetic-geometric mean inequality the sum of the Gauss norms of $\varphi(\gamma_1\cdots \gamma_r)$ satisfies the estimate  \begin{align}\label{free-ineq} ||\varphi'(\gamma_1\cdots \gamma_{r})||\geq 2\sqrt[4]{{65}}^{r-1}\sqrt[]{\big(|(\gamma_1)_{12}|_{\mathbb{E}}\cdots |(\gamma_r)_{12}|_{\mathbb{E}}\big)^{-1}}.\end{align} 

The image of $\varphi'$ is metrically proper. If for a sequence $(g_\ell)$ in $\langle \alpha\rangle \ast \langle \beta \rangle$, $(||\varphi'(g_{\ell})||)$ is bounded by some $M>0$  for every $\ell$, by (\ref{free-ineq}) and up to subsequence, there is $1\leq r\leq 10\log M$ such that $g_{\ell}$ is written in reduced form $g_\ell=\gamma_{1\ell}\cdots \gamma_{r\ell}.$ If $(\gamma_{1\ell})$ is infinite, as $\alpha,\beta\in \mathsf{SO}_2(\mathbb{R})$, there is $b\in \mathbb{N}$ with $\gamma_{1b}^{-1}\gamma_{1\ell}\neq \textup{I}_2$ and $0<|(\gamma_{1b}^{-1}\gamma_{1\ell})_{12}|_{\mathbb{E}}\leq M^{-6}$ for infinitely many $\ell$.  By (\ref{free-ineq}), $$ M^{2}\geq ||\varphi'(g_{\ell})|| \cdot ||\varphi'(g_{b}^{-1})|| \geq ||\varphi'(\gamma_{rb}^{-1}\cdots (\gamma_{1b}^{-1}\gamma_{1\ell})\cdots \gamma_{r\ell})|| \geq \sqrt[]{\big|(\gamma_{1b}^{-1}\gamma_{1\ell})_{12}\big|_{\mathbb{E}}^{-1}}$$ contradicting the choice of $b\in \mathbb{N}$ above. Hence $(\gamma_{1\ell})$ is necssarily finite and inductively we deduce that $(\gamma_{i\ell})$ is finite for every $i=1,\ldots,r$. Hence the image of $\varphi'$ is metrically proper in $G$. Moreover, since $\langle \alpha \rangle$ (resp. $\langle \beta \rangle$) is precompact in $\mathsf{SL}_2(\mathbb{Q}_{13})$ (resp. $\mathsf{SL}_2(\mathbb{Q}_5)$), the projection of $\varphi'(\langle \alpha\rangle \ast \langle \beta \rangle)$ on each factor is indiscrete. \qed}
\end{exmp}

We note that the existence of non-elementary discrete free subgroups in products $\prod_{i\in I}G_i(k_i)$ (where $|I|\geq 2$ is finite and $G_i$ is a simple algebraic group defined over the local field $k_i$) projecting densely on each factor is an open question asked by Yves Benoist (see \cite{BFMvL} and the references therein).

\subsection*{Acknowledgements} I would like to thank Sami Douba,   Miko\l{}aj Fraczyk and Beatrice Pozzetti for helpful discussions. The author was supported by the Dioscuri programme, initiated by the Max Planck Society, jointly managed with the National Science Center in Poland, and mutually funded by Polish Ministry of Education and Science and German Federal Ministry of Education and Research. The author also acknowledges support from  the Simons Laufer Mathematical Sciences Institute in Berkeley, California, during the Spring 2026 semester (National Science Foundation Grant No. DMS2424139). The author would also like to thank the  Max Planck Institute for Mathematics in the Sciences (Leipzig) for its hospitality where most of this work was conducted.

\subsection*{AI Disclosure} The author has not used any form of AI for the mathematical content presented in this article, nor to check or organise any of the proofs provided here.

\bibliographystyle{siam}

\bibliography{biblio.bib}

\end{document}